\documentclass[12pt]{elsarticle}
\usepackage{graphicx}
\usepackage{amssymb}  
\usepackage{epstopdf}
\usepackage{comment}
\usepackage{xcolor}
\usepackage{hyperref}
\usepackage[final]{pdfpages}
\usepackage{frcursive}
\usepackage{bbm}
\usepackage{mathtools}
\usepackage{tabularx}
\usepackage{stmaryrd}
\usepackage{algorithm}
\usepackage{algorithmic}
\usepackage[normalem]{ulem}
\usepackage{dutchcal}
\usepackage{upgreek}
\usepackage{caption}

\usepackage{yfonts}
\usepackage{lscape}
\usepackage{epsfig}
\usepackage{fullpage}
\usepackage[titletoc]{appendix}
\usepackage{fancyhdr}
\usepackage{amsmath,amssymb,xspace}
\usepackage{tabularx}
\usepackage{mathrsfs}

\usepackage{calligra}
\usepackage[T1]{fontenc}

\def\cal{\mathcal}

\newcommand{\bs}[1]{\boldsymbol{#1}}

\newcommand{\R}{\mathbb{R}}

\newcommand{\N}{\mathbb{N}}

\newcommand{\bl}{\begin{law}}
\newcommand{\el}{\end{law}}
\newcommand{\bthm}{\begin{thm}}
\newcommand{\ethm}{\end{thm}}

\newcommand{\jump}[1]{\llbracket #1 \rrbracket}

\newcommand{\mc}[1]{\mathcal{#1} }

\usepackage{chngcntr}
\usepackage{amsthm}

\newfont{\bbb}{msbm10 scaled\magstep1}
\newtheorem{theo}{Theorem}[section]

\newtheorem{rem}{Remark}[section]

\let \leq \leqslant
\let \geq \geqslant
\let \cal \mathcal

{ \par \medskip \par
  \noindent \textit{\textbf{Demonstration\/}} : }{\null \hfill $\Box$ \par }

\makeatletter
\newcommand{\doublewidetilde}[1]{{%
  \mathpalette\double@widetilde{#1}%
}}
\newcommand{\double@widetilde}[2]{%
  \sbox\z@{$\m@th#1\widetilde{#2}$}%
  \ht\z@=.9\ht\z@
  \widetilde{\box\z@}%
}

\allowdisplaybreaks 

\begin{document}

\begin{frontmatter}

\title{A diffusion-free multi-layer neural-network method for multidimensional nonlinear hyperbolic equations}
\author[carl,crm]{Emmanuel LORIN}
\ead{elorin@math.carleton.ca}

\address[carl]{School of Mathematics and Statistics, Carleton University, Ottawa, Canada, K1S 5B6}
\address[crm]{Centre de Recherches Math\'{e}matiques, Universit\'{e} de Montr\'{e}al, Montr\'{e}al, Canada, H3T~1J4}

\author[ott]{Arian NOVRUZI}
\ead{novruzi@uottawa.ca}
\address[ott]{Department of Mathematics and Statistics, University of Ottawa, Ottawa, Canada}

\begin{abstract}
  This paper is devoted to the non-diffusive computation of solutions to $N$-dimensional hyperbolic systems of conservation laws (HCL) and more specifically the accurate approximation of simple waves. We develop the {\it LeafNet} algorithm, which relies on:
i) a reformulation of HCL as a coupled system involving smooth solutions, and
ii) an accurate physics informed algorithms approximating smooth functions. 
An error estimate analysis and two-dimensional numerical experiments are proposed to illustrate the presented strategy.

\end{abstract}


\end{frontmatter}


\section{Introduction}\label{s:intro}
This paper is devoted to the neural network-based solution of multidimensional nonlinear hyperbolic conservation laws (HCL), see \cite{serre,lefloch,smoller}. The use of neural networks for solving (parameterized) PDE has recently become very popular, as beyond the use of automatic differentiation, they allow for the inclusion of external (experimental or numerical) data in the solver.
Among the existing techniques the most popular are the Physics Informed Neural Network (PINN) methods \cite{pinns,pinns2,pinns3}, which can be seen as a generalization of the pioneering work in \cite{lagaris}. 
Let us also cite alternative methods presented in \cite{otherNN,otherNN2,otherNN3}.

In this paper, we are particularly interested in the approximation of simple waves, more precisely shock and rarefaction waves,  which is known to be a difficult numerical problem with standard computational methods, such as finite volume methods \cite{alouges2,roe,god2,god1}. The latter rely on approximate Riemann solvers, and usually produce some artificial diffusion. Increasing the order of approximation allows for a reduction of the artificial diffusion, but also leads to stability issues, requiring the use of slope limiters \cite{MYV}. 
On the other hand, in order to approximate weak non-smooth solutions  alternative methods based on variational formulation of PDEs, such as Deep Ritz \cite{ritz} or Variational-PINN \cite{VPINN}, have been developed. 

PINN-like methods allow for the approximation of smooth solutions to HCL (including rarefaction waves), avoiding to a certain extent stability issues usually due to time-stepping. Indeed, within PINN algorithms, IVP problems are often reformulated and approximated as a BVP, where the initial condition is treated as a boundary condition. Stability issues may however be moved to the optimization algorithm. In order to avoid a null solution, the integration time has to be taken not too large, which represents a drawback of these methods. 

As for the shock waves, their approximation with classical neural network algorithms, 
such as PINN in \cite{pinns,pinns2,pinns3}, is known to be very inaccurate. This is mainly due to the fact that shock waves are non-smooth solutions, defined by the weak formulation of the hyperbolic conservation laws, and that similarly to finite difference methods,  neural network algorithms, such as PINN methods, approximate PDEs in their strong form. 
As a consequence, a direct use of (smooth) neural networks only generates smooth solutions (hence corresponding to regularized shocks).

In this paper, we develop an algorithm allowing to accurately approximate shock waves using neural networks by:
i) approximating smooth solutions, 
ii) approximating the Rankine-Hugoniot jump condition, and 
iii) reconstructing the overall piecewise smooth solution. 
The proposed method can be seen as a multidimensional extension of the algorithm which was developed in \cite{jcp2024} for one-dimensional systems, and which largely relies on the structure of solutions as well as mathematical properties of HCL. 
We will focus on the approximation of shock waves for $N$-dimensional hyperbolic conservation laws, including
their generation from smooth Cauchy data.  \textcolor{black}{ ``Diffusion-free'' refers to the fact that the numerical solution computed by the LeafNet method
captures discontinuous solutions without introducing artificial or numerical diffusion that would otherwise smooth out the waves. Namely, we compute the shock waves as the discontinuity curves/surfaces, which separate regions where the solution is smooth and HCL are satisfied in a classical sense, 
and where the Rankine-Hugoniot condition is hence accurately satisfied.} 
We also will design self-similar neural networks allowing for an accurate approximation of rarefaction waves. Finally we will extend our algorithm to the systems of HCL. The extension of the proposed strategy to more complex problems, such as wave-interactions, overturning or closed surfaces, will covered in a forthcoming paper.

The paper is organized as follows. 
In Section \ref{s:1xN,HCL}, we derive and analyze our method for $N$-dimensional scalar conservation laws.
Namely, in Subsection \ref{ss:1xN,shock,method}, we discuss the general strategy for $N$-dimensional scalar hyperbolic equations using the mathematical structure of the weak solutions.
More precisely, in Subsection \ref{sss:1xN,shock,alg} we present the associated algorithm.
In Subsection  \ref{sss:1xN,shock,theo} we propose a mathematical justification of our
algorithm, followed by a detailed mathematical analysis of the error for shock wave solutions
in Subsection \ref{sss:1xN,shock,error}.
In Subsection \ref{ss:1xN,shock,form} we discuss the formation of shock waves and in Subsection \ref{ss:1xN,rw} an algorithm for generating rarefaction waves is derived.
The extension our method to $N$-dimensional systems of HCL is presented in Section \ref{s:mxN,HCL}. Section \ref{s:numerics} is devoted to numerical experiments, and concluding remarks are proposed in Section \ref{s:conclusion}.

We finish this introduction by recalling some basic concepts on neural networks in $N$-dimensions.
Let $l\in\mathbb N$, $n_i\in\mathbb N$, $i=0, 1,\ldots, l$, with $n_0=N$,
and $\sigma:\mathbb R\mapsto\mathbb R$ a given nonlinear function, called activation function.	
For $\theta \in  \prod_{i=1}^l \mathbb R^{n_i\times (n_{i-1}+1)}$, we write
	\begin{eqnarray*}
		\theta&=&(\theta^i),\quad \theta^i=(w^i,b^i),\quad i=1,\ldots,l,\;\;\mbox{\it where}\\
		&&w^i(\cdot,\cdot)\in \mathbb R^{n_i\times n_{i-1}},\\
		&&b^i(\cdot)\in \mathbb R^{n_i}.
	\end{eqnarray*}
	A (fully connected) network in $\R^N$ with architecture $A=[n_0,n_1,\ldots,n_l]$ is a function of the form 
\begin{subequations}
	\begin{align}
	{\bf u}&:(\theta;x)\in \prod_{i=1}^l \mathbb R^{n_i\times (n_{i-1}+1)}\times \mathbb R^N 
	\mapsto
	{\bf u}(\theta;x)\in \mathbb R^{n_l},
	\\
	{\bf u}(\theta;x)
	&=
	w^l\cdot\sigma(w^{l-1}\cdot(\cdots \sigma(w^2\cdot\sigma(w^1\cdot x + b^1)+b^2)\cdots)+b^{l-1})+b^l,
	\end{align}
	\end{subequations}
where $\sigma$ applies component-wise.
Given an architecture $A$, the set of all parameters $\theta$ associated to  $A$ is denoted by 
$\Theta_A$. The set of all network functions associated to parameters in $\Theta_A$ is denoted by 
${\bf N}_A$. All along this paper, the network functions will be denoted by boldface letters.

\section{Scalar conservation laws in $N$ dimensions}\label{s:1xN,HCL}
In this section we consider an $N\geq 1$ dimensional nonlinear scalar hyperbolic conservation law
in $Q:=\Omega\times [0,T]$ for some (un)bounded open smooth set $\Omega \subseteq \R^N$. 
Denoting $\nabla=(\partial_1,\ldots,\partial_N)=(\partial_{x_1},\ldots,\partial_{x_N})$, the 
problem we consider is
\begin{subequations}\label{e:CL(N,1)}
\begin{align}
    \partial_t u + f(u)\cdot\nabla u
    & = 0,\hspace*{10mm} \textrm{ in } Q, \\
    u(\cdot,0) & =  u_0(\cdot),\hspace*{5mm} \textrm{ on } \Omega,
\end{align}
\end{subequations}
where
$u: (x,t):=(x_1,\ldots,x_N,t)\in Q \rightarrow u(x,t) \in \R$,
$f(u)=(f_1(u),\ldots,f_N(u))=:F'(u)$, $F\in C^2(\R;\R^N)$.

We impose incoming flux boundary conditions - that is boundary conditions are imposed where characteristic lines are incoming. For simplicity of presentation, we will not discuss these boundary conditions and we refer to \cite{jmg} for details. 
For simplicity of presentation, here we will consider $\Omega=\prod_{i=1}^N(a_i,b_i)$, or $\Omega=\R^N$.
The latter  case keeps the full generality of the problem, if we extend smoothly $u_0$ to a function with compact support, and reducing the time $T$.

\subsection{Shock waves and the principle of the LeafNet method}\label{ss:1xN,shock,method}
Let $\Omega=\prod_{i=1}^N(a_i,b_i)$.
Given two functions $u_0^\pm\in C^1(\Omega)$, and a $C^1$-surface $\Gamma_0\subset\Omega$, parameterized by ${g}_0\in C^1(\Omega')$, $\Omega':=(a_1,b_1)\times\cdots\times(a_{N-1},b_{N-1})$, 
\begin{eqnarray*}
  \Gamma_0=
  \left\{
  (x',x_N) \in \Omega,\;\, 
   x_N={g}_0(x'),
  \quad x'=(x_1,\ldots,x_{N-1})\in\Omega'\right\},
\end{eqnarray*}
where the direction $x_N$ is chosen arbitrarily, without loss of generality. We define $\Omega^\pm\subset \Omega$ by
\begin{equation}\label{e:Omegapm}
\Omega^+=\{(x',x_N)\in\Omega,\; x_N>g_0(x')
\},\quad
\Omega^-=\{(x',x_N)\in\Omega,\; x_N<g_0(x')
\},
\end{equation}
and the $C^1(\Omega)$-piecewise function $u_0$ by
\begin{equation}\label{e:u0=u0+,u0-}
u_0(x)=u_0(x',x_N)=
\left\{
\begin{array}{ll}
u_0^+(x),& x_N>g_0(x'),\\
u_0^-(x),& x_N<g_0(x').
\end{array}
\right.
\end{equation}
The shock of \eqref{e:CL(N,1)} with $u_0$ given by \eqref{e:u0=u0+,u0-} is identified 
as the surface $\Gamma$ in $Q$, where the Rankine-Hugoniot condition is satisfied.  
We assume that the evolution of the shock wave of  
\eqref{e:CL(N,1)} is parameterized  by $x_N={g}(x',t)$, i.e.
\begin{eqnarray}\label{e:sigma(N)}
\Gamma& = & \big\{(x',{g}(x',t),t) \in Q\,\, : \,\, 
(x',t) \in Q':=\Omega'\times(0,T)\big\}.
\end{eqnarray}
Our objective is in particular to identify the function ${g}$.
Let $\nu=(\nu_1,\ldots,\nu_N,\nu_t)$ be the unit normal vector on $\Gamma$ oriented on the direction of 
$x_N$ axis given by
\begin{equation}
\nu
=
\frac{\nabla_{(x,t)}(x_N-{g}(x',t)}{|\nabla_{(x,t)}(x_N-{g}(x',t))|}
= 
\frac{(-\partial_1{g},\ldots,-\partial_{N-1}{g},1,-\partial_t{g})}
     {\sqrt{1+|\partial_1{g}|^2+\cdots+|\partial_{N-1}{g}|^2 + |\partial_{t}{g}|^2}},          \label{e:nu(N)}
\end{equation}
The Rankine-Hugoniot jump condition on $\Gamma$ in terms of $u^\pm$ reads
\begin{eqnarray}
\nu_t\jump{u} + \jump{F(u)}\cdot(\nu_1,\ldots,\nu_N)
&=& 
0,	\label{e:RH(N,1)-1}
\end{eqnarray}
where in general, for any function $v$, the jump operator $\jump{\cdot}$ on $\Gamma$ is defined by
\begin{equation}\label{e:jump(v)}
\jump{v}=\lim_{h\to0}(v(x',g(x')+h^2)-v(x',g(x')-h^2)).
\end{equation} 
Then by rearranging \eqref{e:RH(N,1)-1}, we deduce that ${g}$ satisfies 
\begin{subequations}\label{e:RH(N,1)}
\begin{align}
\partial_t{g} 
+ 
\frac{\jump{F_1(u)}}{\jump{u}}\partial_1{g}
+\cdots+
\frac{\jump{F_{N-1}(u)}}{\jump{u}}\partial_{N-1}{g}
&=:					\nonumber\\
\partial_t{g} 
+
a_1(g,x',t)\partial_1g
+\cdots+
a_{N-1}(g,x',t)\partial_{N-1}g
&=
b(g,x',t)				  \label{e:RH(N,1),i}\\
&:=
\frac{\jump{F_N(u)}}{\jump{u}},		\nonumber\\
{g}(\cdot,0)
  &=
  {g}_0(\cdot) \, .	 	 \label{e:RH(N,1),b}
\end{align}
\end{subequations}
Hereafter we rewrite the jumps $\jump{\cdot}$ in \eqref{e:RH(N,1)} in a form that will be used 
for the analysis and the error estimate in the following sections.
First, we consider the following problems (with appropriate incoming flux boundary conditions, or no boundary conditions if $u_0^\pm$ are extended smoothly to functions with compact support),
\begin{subequations}\label{e:upm(1,N)}
\begin{align}
    \partial_t u^\pm + f(u^\pm)\cdot \nabla u^{\pm} & =  0, \,\,\,\;\; \textrm{ in } Q, \label{e:upm(1,N),i}
    \\
    u^\pm(\cdot,0) & =  u_0^\pm, \,\,\, \textrm{ on } \Omega.	\label{e:upm(1,N),b}
\end{align}
\end{subequations}
Subject to taking $T$ smaller, we assume that  there is no shock formation for both IBVPs \eqref{e:upm(1,N)}
for $t\leq T$. 
Following \cite[Theorem 3.1]{majda1984},  this holds if both $\nabla\cdot f(u_0^\pm)$ take negative values and $\nabla\cdot f(u_0^\pm)$ are bounded. In such a case, $T$ can be defined as
\begin{equation}\label{e:noshock(N,1)} T=-\frac{1}{\min\{\nabla\cdot f(u_0^\pm(\alpha)),\; \alpha\in\mathbb R^N\}},\quad
\mbox{where}\;\; 
\nabla\cdot f(u_0^\pm)=
\sum_{i=1}^N f'_i(u_0^\pm)\partial_iu_0^\pm,
\end{equation}
From the uniqueness of classical solutions  to \eqref{e:CL(N,1)} we have $u=u^+$ above $\Gamma$ and $u=u^-$ below $\Gamma$ (in the $x_N$ direction).
Then for $i=1,\ldots,N$ we have
\begin{eqnarray*}
\jump{F_i(u)}
&=&
F_i(u^+(x',{g}(x',t),t))-F_i(u^-(x',{g}(x',t),t))
\\
&=&
(u^+(x',{g}(x',t),t)-u^-(x',{g}(x',t),t))\\
&&
\cdot
\int_0^1 f_i(u^-(x',{g}(x',t),t) + 
             s(u^+(x',{g}(x',t),t) - u^-(x',{g}(x',t),t)))ds
\\
&=&
\jump{u}
\int_0^1 f_i(u^- + s\jump{u})ds
\\
&=:&
\jump{u}
\left\{\begin{array}{ll}a_i(x',t),&i=1,\ldots,N-1,\vspace*{2mm}\\b(x',t),&i=N.\end{array}\right.
\end{eqnarray*}
This allows to rewrite the coefficients $a_i$ and $b$ in \eqref{e:RH(N,1)} as follows
\begin{subequations}\label{e:ai,b}
\begin{align}
&a=(a_1,\ldots,a_{N-1}),				\nonumber\\
&a_i(g,x',t)=
\int_0^1 f_i(u^-(x',{g},t) + s(u^+(x',{g},t) - u^-(x',{g},t)))ds,
\\
&b(g,x',t)=\int_0^1 f_N(u^-(x',{g},t) + s(u^+(x',{g},t) - u^-(x',{g},t)))ds.
\end{align}
\end{subequations}
We will use \eqref{e:RH(N,1)} with the form \eqref{e:ai,b} of coefficients $a$, $b$ for estimating the error of the method.
It will be helpful in the following to write \eqref{e:RH(N,1)} in operator form as follows
\begin{equation}\label{e:K[g]=0}
K[g]=(K_i[g],K_b[g])
:=
(\partial_t{g} + a\cdot \nabla_{x'} g - b, g(\cdot,0)- g_0(\cdot))=(0,0),
\end{equation}
where $a$ and  $b$ are given by \eqref{e:ai,b}.

Hence, ``solving'' the shock $\Gamma$ of \eqref{e:CL(N,1)} is equivalent to solve the $N-1$ dimensional 
IVP \eqref{e:RH(N,1)}. The method, that we will refer as {\it LeafNet method}, consists in:
%
\begin{equation}\label{m:shock-wave,Nx1}
\left.
\begin{array}{lp{130mm}}
1)& 
\mbox{\it Solving (independently) the two IBVP \eqref{e:upm(1,N)},
 which provide $u^-$ and $u^+$ in $Q$}.\\
2)&
\mbox{\it Solving the Rankine-Hugoniot \eqref{e:RH(N,1)}, which provides ${g}$ in $Q'$}.
\\
3)& 
\mbox{\it Reconstructing the solution from the initial IBVP \eqref{e:upm(1,N)} for $(x',x_N,t)\in Q$,}
\\
& \mbox{\it as follows}
    \begin{eqnarray*}
      u(x',x_N,t)=
      \left\{
      \begin{array}{ll}
        u^-(x',x_N,t), & \textrm{ if } x_N<{g}(x',t)\, ,\\
         u^+(x',x_N,t), & \textrm{ if } x_N>{g}(x',t)\, .       
        \end{array}
      \right.
      \end{eqnarray*}
\end{array}
\right.
\end{equation}
Such a construction of the shock waves is presented in \cite{majda1984} for $N=1,2$. 
Here we extend it in dimension $N$. Furthermore, we will provide precise $C^1$ estimates for
the solutions $u^\pm$, which will be useful for the error estimate of the network approximation 
of the shock wave.

\subsubsection{Neural network algorithm: the LeafNet algorithm}\label{sss:1xN,shock,alg}
The leaflnet method, presented in Subsection \ref{ss:1xN,shock,method}, see \eqref{m:shock-wave,Nx1}, benefits from the accurate approximation of classical solutions by PINN algorithms, hence avoiding spurious regularization of discontinuous waves. 
We implement it by using a combination of PINN algorithms. This algorithm will be referred as the {\it LeafNet} algorithm. 

We introduce two neural networks 
${\bf u}^\pm(x,t)$, $x=(x_1,\ldots,x_N)$ for solving \eqref{e:upm(1,N)}.
In addition, we introduce one neural network ${\bf g}(x',t)$ for solving \eqref{e:RH(N,1)}. 
More precisely we proceed as follows.

\begin{enumerate}
\item 
{\it Approximation of \eqref{e:upm(1,N)}}. We introduce the operator 
\begin{equation}\label{e:Lpm(N,1)}
L^\pm[u]=(L^\pm_i[u],L^\pm_b[u])
:=
(\partial_t u+ f(u)\cdot\nabla u,u(\cdot,0)-u_0^\pm).
\end{equation}
Given an architecture $A$ in $Q$, with parameters $\Theta_A$,
for $k,m\in\N_0$ and ${\bf u}\in{\bf N}_A$ define
\begin{subequations}
\begin{align}
\bs{\cal L}^\pm[{\bf u}]
&=      
\sum_{\substack{\alpha\in\N_0^N,\beta\in\N_0\\ |\alpha|+\beta\leq k}}
\int_0^T
\int_\Omega
|D^\alpha_xD^\beta_tL^\pm_i[{\bf u}]|^2 dxdt
+
\sum_{\substack{\alpha\in\N_0^N,|\alpha|\leq m}}
\int_\Omega
|D^\alpha_xL^\pm_b[{\bf u}]|^2 dx.			\label{e:cLpm(N,1)}
\end{align}
\end{subequations}
We refer to \cite{de2021approximation,de2024error} for discussions on the use of Sobolev norms to define loss functions.
Next we look for ${\bf u}^\pm\in{\bf N}_A$ the solution to
\begin{equation}
{\bf u}^\pm
=\textrm{argmin}\{\bs{\cal L}^\pm[{\bf u}],\;\; {\bf u}\in{\bf N}_A\}.    \label{e:argmin(cLpm[u])}
\end{equation}
Note that in \eqref{e:argmin(cLpm[u])}, the optimization is made with respect to the weights 
$\theta\in\Theta_A$, and  ${\bf u}^\pm$ represents the network which corresponds to the optimal 
parameters $\theta^\pm\in\Theta_A$.
We will show that under specific conditions, ${\bf u}^\pm$ approximates $u^\pm$.

\item
{\it Solution to the IVP \eqref{e:RH(N,1)}}.
We consider an architecture $A'$ in $\Omega'$, 
and denote by $\Theta_{A'}$, resp. ${\bf N}_{A'}$, its associated set of parameters and network functions.
For ${\bf h}\in {\bf N}_{A'}$ we define
\begin{equation}\label{e:bfK[h]}
{\bf K}[{\bf h}]
=
({\bf K}_i[{\bf h}],{\bf K}_b[{\bf h}])
=
(\partial_t{\bf h} + {\bf a}\cdot \nabla_{x'}{\bf h} - {\bf b}, {\bf h}(\cdot,0)- g_0(\cdot)),
\end{equation}
where
\begin{subequations}\label{e:bfai,bfb}
\begin{align}
&{\bf a}=({\bf a}_1,\ldots,{\bf a}_{N-1}),				\nonumber\\
&{\bf a}_i({\bf h},x',t)
=
\int_0^1 f_i({\bf u}^-(x',{\bf h},t) + s({\bf u}^+(x',{\bf h},t) - {\bf u}^-(x',{\bf h},t)))ds,
\\
&{\bf b}({\bf h},x',t)
=
\int_0^1 f_N({\bf u}^-(x',{\bf h},t) + s({\bf u}^+(x',{\bf h},t) - {\bf u}^-(x',{\bf h},t)))ds,
\end{align}
\end{subequations}
and $u^\pm$ are defined in \eqref{e:upm(1,N)}.
For $p,q\in\N_0$ and ${\bf h}\in{\bf N}_{A'}$ define
\begin{subequations}\label{e:cK[h]}
\begin{align}
\bs{\cal K}[{\bf h}]
&=      
\sum_{\substack{\alpha\in\N_0^N,\beta\in\N_0\\ |\alpha|+\beta\leq p}}
\int_0^T
\int_{\Omega'}
|D^\alpha_xD^\beta_t {\bf K}_i[{\bf h}]|^2 dxdt
+
\sum_{\substack{\alpha\in\N_0^N,|\alpha|\leq q}}
\int_{\Omega'}
|D^\alpha_x {\bf K}_b[{\bf h}]|^2 dx \, ,		
\end{align}
\end{subequations}
and define ${\bf g}\in{\bf N}_{A'}$  by
\begin{equation}
    	{\bf g} = \textrm{argmin}
    	\{\bs{\cal K}[{\bf h}],\;\; {\bf h}\in {\bf N}_{A'}\},
\label{e:argmin(cK[h])}
\end{equation}
\item 
Reconstruction of the approximate solution ${\bf u}$ to \eqref{e:CL(N,1)}, for $(x',x_N,t)\in Q$
    \begin{eqnarray}\label{algo:rec}
      {\bf u}(x',x_N,t)=
      \left\{
      \begin{array}{ll}
        {\bf u}^-(x',x_N,t), & \textrm{ if } x_N<{\bf g}(x',t)\, ,\\
         {\bf u}^+(x',x_N,t), & \textrm{ if } x_N>{\bf g}(x',t)\, .       
        \end{array}
      \right.
    \end{eqnarray}
\end{enumerate}  
In this paper, we have not imposed any hard constraints forcing the solver to only approximate entropic waves only. Entropy conditions, such as (23), could easily be included within the loss but would deteriorate the training of the neural networks. Let us add, that entropy questions was shortly addressed in \cite{jcp2024}.
\subsubsection{Theoretical justification of Algorithm \ref{sss:1xN,shock,alg}}\label{sss:1xN,shock,theo}
In this section we will justify rigorously the LeafNet method \eqref{m:shock-wave,Nx1}.
The existence part of the following result has been presented in \cite{majda1984} in dimension two.
We extend it in dimension $N$ and provide some 
$L^\infty$-estimates which will be important for estimating the error on approximate shock waves.

To avoid the technicalities related to the boundary conditions, all along this section we
assume that $u^\pm$ are with compact support
and we take $\Omega=\R^N$, $\Omega'=\R^{N-1}$. 
\begin{theo}\label{th:u,Sigma(N,1)}
Let $F=(F_1,\ldots,F_N)\in C^2(\R,\R^N)$, $f=(f_1,\ldots,f_N):=F'$, with $f_N$ increasing,
${g}_0\in C^1_b(\R^{N-1})$ and 
$\Gamma_0=\{(x',{g}_0(x')),\; x'\in\R^{N-1}\}$ the graph of ${g}_0$ with normal vector 
$\nu^0=(\nu^0_1,\ldots,\nu^0_N)$ oriented toward $x_N$ axis.
Let also $u_0^+,u_0^-\in C^1(\R^N)$ with compact support satisfying the following entropy condition (see \cite{majda1984})
\begin{equation}\label{e:shockhere}
f_N(u_0^-(x))
>
\frac{\jump{F_N(u_0(x))}}{\jump{u_0(x)}}
>
f_N(u_0^+(x)),\quad \forall x\in\Gamma_0,
\end{equation}
where $\jump{\varphi(u)}:=\varphi(u_0^+)-\varphi(u_0^-)$ on $\Gamma_0$ for any function $\varphi$,
see \eqref{e:jump(v)}, and define $u_0$ as in \eqref{e:u0=u0+,u0-}.
%
If $T>0$ is small enough then the following statements hold:
\\
(i)
Problems \eqref{e:upm(1,N)} have unique solutions $u^+, u^-\in C^1(\R^N\times[0,T])$. Furthermore,
\begin{eqnarray}
\hspace*{-10mm}
\|u^\pm\|_{L^\infty(\R^N\times[0,T])} 
&\leq &
\|u_0^\pm\|_{L^\infty(\R^N)}:=\rho,
\label{e:|upm|Linf}
\\
\hspace*{-10mm}
\|\partial_tu^\pm\|_{L^\infty(\R^N\times[0,T])} 
+
\|\nabla u^\pm\|_{L^\infty(\R^N\times[0,T])} 
&\leq &
\|\nabla u_0\|_{L^\infty(\R^N)} \nonumber\\ 
&+&
C(\|\partial_zf\|_{L^\infty(B_\rho\times\R^N\times\R)},\|\nabla_xf\|_{L^\infty(B_\rho\times\R^N\times\R)}).
\label{e:|Dtxupm|Linf}
\end{eqnarray}
(ii)
The problem \eqref{e:RH(N,1)} has a unique piecewise $C^1(\R^{N-1}\times[0,T])$ weak solution,
and \eqref{e:CL(N,1)} has a unique piecewise $C^1(\R^N\times[0,T])$ solution $u$, with a $C^1$-discontinuity surface  $\Gamma$ with equation $x_N={g}(x',t)$,  which can be constructed as in 3), 
\eqref{m:shock-wave,Nx1}.
\end{theo}

The proof of this theorem is based on the following general result  for the quasi-linear first order PDE problems:
\begin{subequations}\label{e:QL1oPDE}
\begin{align}
    \partial_t u + \sum_{i=1}^N f_i(u,x,t)\partial_i u &=r(u,x,t), \,\,\,\;\; \textrm{ in } \R^N\times (0,T), 
    \label{e:QL1oPDE,i}
    \\
    u(\cdot,0) & = u_0(\cdot), \,\,\, \textrm{ on } \R^N \, ,
    \label{e:QL1oPDE,b}
\end{align}
\end{subequations}
where 
$u: (x,t)\in \R^N\times[0,T]\mapsto u(x,t) \in \R$, and $f$, $r$ satisfy:
\begin{subequations}\label{e:f->C1}
\begin{align}
&\;\;f=(f_1,\ldots, f_N),\;\;
f_i=f_i(z,x,t)\in  C^1(\R\times\R^N\times\R),\\
\forall\rho>0,\;\exists C(\rho)\geq0:\;\;&\;\; 
\|f\|_{L^\infty(B_\rho\times\R^N\times\R)}
+
\|\nabla_{(z,x)}f\|_{L^\infty(B_\rho\times\R^N\times\R)}\leq C(\rho)<\infty,
\end{align}
\end{subequations}
and 
\begin{subequations}\label{e:r->C1}
\begin{align}
&\;\; r=r(z,x,t)\in C^1(\R\times\R^N\times\R), 		\label{e:r->C1,1}\\
\exists C_r\geq0:&\;\;\;
\|r(z,\cdot,\cdot)-r(0,\cdot,\cdot)\|_{L^\infty(\R^N\times\R)}\leq C_r|z|,
				\label{e:r->C1,2}\\
\forall \rho>0,\; \exists C(\rho)\geq0:&\;\;\;
\|r\|_{L^\infty(B_\rho\times\R^N\times\R)}+
\|\nabla_{(z,x)} r\|_{L^\infty(B_\rho\times\R^N\times\R)}\leq C(\rho)<\infty,	\label{e:r->C1,3}
\end{align}
\end{subequations}
where $B_\rho$ is the ball in $\R$ with radius $\rho$ and center at the origin and $C(\rho)$ is  
a constant depending on $\rho$ (and $f$ and  $r$, but having $f$ and $r$ fixed we omit the dependence on $f$ and $r$).

The following result related to the problem \eqref{e:QL1oPDE} is very important for proving Theorem \ref{th:u,Sigma(N,1)}. 
The result is classical, but as we did not find it firmly stated in the literature, 
for sake of completeness and especially for establishing the estimates that we will need,
we present it in detail. 
\begin{theo}\label{th:QL1oPDE,sol}
Let $u_0\in C^1(\R^N)$ with compact support,
and $f$, resp. $r$, as in \eqref{e:f->C1}, resp \eqref{e:r->C1}.
Then there exists $T>0$ such that \eqref{e:QL1oPDE} has a unique classical solution 
$u\in C^1(\R^N\times[0,T])$ satisfying 
\begin{eqnarray}
\hspace*{-5mm}
\|u\|_{L^\infty(\R^N\times[0,T])} 
&\leq &
(\|u_0\|_{L^\infty(\R^N)} + T\|r(0,\cdot)\|_{L^\infty(\R^N\times\R)})e^{TC_r}=:\rho,
\label{e:|u|Linf}
\\
\hspace*{-5mm}
\|\partial_tu\|_{L^\infty(\R^N\times[0,T])} 
+
\|\nabla u\|_{L^\infty(\R^N\times[0,T])} 
&\leq &
\|\nabla u_0\|_{L^\infty(\R^N)} 
\label{e:|Dtxu|Linf}\\
&+&
C(\|\nabla_{(z,x)}f\|_{L^\infty(B_\rho\times\R^N\times\R)},
  \|\nabla_{(z,x)}r\|_{L^\infty(B_\rho\times\R^N\times\R)}),
\nonumber
\end{eqnarray}
where the constant $C(\cdot,\cdot)$ can be estimated by using  \eqref{e:|Da(X-I)|Linf+|Daz|Linf}, \eqref{e:|DxXm1|Linf}.
\end{theo}
\noindent
{\bf Proof}.
The characteristic equations of \eqref{e:QL1oPDE} are given by
\begin{subequations}\label{e:char-eqs,X,z}
\begin{align}
X(t;\alpha)
&
= \alpha+\int_0^t f(z(s;\alpha),X(s,\alpha),s)ds,\;\; t\geq0,\;\;
&&X(0,\alpha)= \alpha,\,\,\alpha\in\R^N,		\label{e:char-eqs,X}
\\
z(t;\alpha)&
=
u_0(\alpha)+\int_0^t r(z(s;\alpha),X(s,\alpha),s)ds,\;\; t\geq0,\;\;
&&
z(0,\alpha)=u_0(\alpha).	\label{e:char-eqs,z}
\end{align}
\end{subequations}
where  $X(t;\alpha)=(X_1(t;\alpha),\ldots,X_N(t;\alpha))$.
By a standard fixed point argument, it is easy to show that the equations \eqref{e:char-eqs,X,z} 
have a global $C^1$ unique solution. 
As $f$ and $r$ are $C^1$, form  classical results we know that $X(t;\alpha)$ and $z(t;\alpha)$ are 
$C^1$ in $(t,\alpha)\in\R\times\R^N$ and $C^2$ in $t\in\R$.

Now we construct the solution $u(x,t)$ by using $z(t;\alpha)$ and $X(t;\alpha)$.
For this we need to show that the map 
$(\alpha,t)\in \R^N\times[0,T]\mapsto (X(t;\alpha),t)\in \R^N\times[0,T]$ is $C^1$-invertible,
where $T>0$ is small enough.
This requires estimates for $X$, $z$ and their derivatives near $(\alpha,0)\in\R^{N+1}$.

From \eqref{e:char-eqs,z} and assumptions \eqref{e:r->C1,2}, \eqref{e:r->C1,3} we get
\begin{eqnarray*}
|z(t,\alpha)|
&\leq&
\|u_0\|_{L^\infty(\R^N)} + 
t\|r(0,\cdot)\|_{L^\infty(\R^N\times\R)}
+
C_r
\int_0^t|z(s,\alpha)|ds,
\end{eqnarray*}
which combined with Gr\"onwall's inequality gives
\begin{eqnarray}
\|z\|_{L^\infty([0,T]\times\R^N)}
&\leq&
(\|u_0\|_{L^\infty(\R^N)} + T\|r(0,\cdot)\|_{L^\infty(\R^N\times\R)})e^{TC_r}
=:\rho. 				\label{e:|z|Linf}
\end{eqnarray}
This estimate and \eqref{e:char-eqs,X} imply
\begin{equation}
\|X-I_N\|_{L^\infty([0,T]\times\R^N)}
\leq
T\|f\|_{L^\infty(B_\rho\times\R^N\times\R)}.		\label{e:|X|Linf}
\end{equation}
The equations \eqref{e:char-eqs,X,z} in combination with \eqref{e:|z|Linf} yield
\begin{subequations}
\begin{align}
\|\partial_t X\|_{L^\infty([0,T]\times\R^N)}
&\leq
\|f\|_{L^\infty(B_\rho\times\R^N\times\R)},	\label{e:|DtX|Linf}
\\
\|\partial_t z\|_{L^\infty([0,T]\times\R^N)}
&\leq
\|r\|_{L^\infty(B_\rho\times\R^N\times\R)}. 				\label{e:|Dtz|Linf}
\end{align}
\end{subequations}
By differentiating with respect to $\alpha$ the equations 
\eqref{e:char-eqs,X,z} gives
\begin{eqnarray*}
\left[
\begin{array}{c}{\partial_{\alpha_i} X(t;\alpha)}\\\partial_{\alpha_i} z(t;\alpha)\end{array}
\right]
&=&
\left[
\begin{array}{c}
e^i\\
\partial_{\alpha_i} u_0(\alpha)
\end{array}
\right]
\\
&+&
\int_0^t 
\left[
\begin{array}{cc}
\partial_xf(z(s;\alpha),X(s;\alpha),s),
&
\partial_zf(z(s;\alpha),X(s;\alpha),s)
\\
\partial_xr(z(s;\alpha),X(s;\alpha),s),
&
\partial_zr(z(s;\alpha),X(s;\alpha),s)
\end{array}
\right]
\cdot
\left[
\begin{array}{c}{\partial_{\alpha_i} X(t;\alpha)}\\\partial_{\alpha_i} z(t;\alpha)\end{array}
\right]
ds,
\end{eqnarray*}
with $e^i$ the $i$th vector of the canonical base in $\R^N$.
It implies that 
\begin{eqnarray*}
&&|\partial_{\alpha_i} X - e^i|+|\partial_{\alpha_i} z-\partial_{\alpha_i}u_0|\\
&\leq&
t
\left(
\sum_{h=f,r}\|\partial_xh\|_{L^\infty(B_\rho\times\R^N\times\R)}
+
\sum_{h=f,r}\|\partial_zh\|_{L^\infty(B_\rho\times\R^N\times\R)}\|\nabla u_0\|_{L^\infty(\R^N)}
\right)
\\
&+&
\sum_{\substack{h=f,r\\k=z,x}}
\|\partial_kh\|_{L^\infty(B_\rho\times\R^N\times\R)}
\int_0^t 
(|\partial_{\alpha_i} X - e^i|+|\partial_{\alpha_i} z|)ds,
\end{eqnarray*}
which, after using Gr\"onwall's inequality, gives for all $i=1,\ldots,N$
\begin{eqnarray}
&&\|\partial_{\alpha_i} X - e^i\|_{L^\infty([0,T]\times\R^N)}
+
\|\partial_{\alpha_i} z-\partial_{\alpha_i}u_0\|_{L^\infty([0,T]\times\R^N)}
\nonumber\\
&\leq&
T
\left(
\sum_{h=f,r}\|\partial_xh\|_{L^\infty(B_\rho\times\R^N\times\R)}
+
\sum_{h=f,r}\|\partial_zh\|_{L^\infty(B_\rho\times\R^N\times\R)}\|\nabla u_0\|_{L^\infty(\R^N)}
\right)\nonumber\\
&&
\exp\left(T\sum_{\substack{h=f,r\\k=z,x}}\|\partial_kh\|_{L^\infty(B_\rho\times\R^N\times\R)}\right).
\label{e:|Da(X-I)|Linf+|Daz|Linf}
\end{eqnarray}
Now we consider the map
\[
H:\R^N\times\R\mapsto \R^N\times\R,\quad
(x,t)=H(\alpha,t):=(X(t;\alpha),t) \, .
\]
Note that $H$ is $C^1$ and
\[
\nabla_{(\alpha,t)} H(t,\alpha)
=
\left[\begin{array}{c|c}\nabla_\alpha X&\partial_t X\\\hline 0&1\end{array}\right],
\]
which in view of \eqref{e:|Da(X-I)|Linf+|Daz|Linf} gives
\begin{eqnarray*}
{\rm det}(\nabla_{(\alpha,t)}H(\alpha,t))
&=&
{\rm det}(\nabla_\alpha X(t,\alpha,t))
=
{\rm det}(I_N+(\nabla_\alpha X(t,\alpha,t)-I_N))
\\
&\geq&
1-P(T),
\end{eqnarray*}
where $P(T)$ is a homogeneous polynomial of $T$
(its coefficients depend on $\|\partial_k h\|_{L^\infty(B_\rho\times\R^N\times\R)}$,
$h=f,r$, $k=z,x$).
Therefore, given any $\epsilon>0$, there exists $T$ small enough such that 
${\rm det}(\nabla_{(\alpha,t)} H(X(t;\alpha),t))\geq \epsilon$, for all 
$(t,\alpha)\in[0,T]\times\R^N$.
From Hadamard theorem 
it follows that 
$H$ is invertible in $[0,T]\times\R^N$, and we denote by $H^{-1}(x,t)=(X^{-1}(x,t),t)$ its inverse.
Furthermore, we can obtain a bound for $\nabla_{(x,t)}X^{-1}$, which will be used for estimating $\nabla u$,
as follows. Note  that
\[
\nabla_{(x,t)} H^{-1}(x,t)
=
\left[\begin{array}{c|c}\nabla_x X^{-1}&-\nabla_x X^{-1}\cdot\partial_t X\\\hline 0&1\end{array}\right] \, .
\]
Without loss of generality we assume 
\begin{equation}
|||I_N-\nabla_\alpha X|||_{L^\infty([0,T]\times\R^N)}
:=
\max\{N |||I_N-\nabla_\alpha X(t,\alpha)|||,\; (t,\alpha)\in[0,T]\times\R^N\}
<1,
\end{equation}
which in view of \eqref{e:|Da(X-I)|Linf+|Daz|Linf} is possible if we take $T$ smaller.
Here, for any matrix $A=(a_{i,j})\in\R^{N\times N}$,
we used the sub-multiplicative matrix norm
$|||A|||=N\max\{|a_{i,j}|,\; i,j=1,\ldots,N\}$, which satisfies $|||A\cdot B|||\leq |||A|||\cdot|||B|||$.
Then 
\begin{eqnarray}
\nabla_x X^{-1}(x,t)
&=&
(\nabla_\alpha X(t,\alpha))^{-1}
=
(I_N-(I_N-\nabla_\alpha X(t,\alpha)))^{-1},
\nonumber\\
&=&
\sum_{n=0}^\infty 
(I_N-\nabla_\alpha X(t,\alpha))^n,\quad x=X(t,\alpha);\;\; \mbox{\it hence}
 	\label{e:DxXm1}\nonumber
\\
||\nabla_x X^{-1}(x,t)\|_{L^\infty(\R^N\times[0,T])}
&:=&
\{\max\{\|\partial_{x_i}X^{-1}_j(x,t)\|_{L^\infty(\R^N\times[0,T])},\;\; i,j=1,\ldots,N\}	\nonumber\\
&\leq&
\frac{1}{N(1-N\|I_N-\nabla_\alpha X\|_{L^\infty([0,T]\times\R^N)})}.	\label{e:|DxXm1|Linf}
\end{eqnarray}

We define $u(x,t)=z(t;X^{-1}(x,t))$, $(x,t)\in \R^N\times[0,T]$.
It is classical to show that $u\in C^1(\R^N\times[0,T])$ and it is the unique classical solution of \eqref{e:QL1oPDE}.

The estimate \eqref{e:|u|Linf} follows from \eqref{e:|z|Linf}.
The estimate   for $\nabla u$ in \eqref{e:|Dtxu|Linf} follows from
$\partial_{x_i} u=\sum_{j=1}^N\partial_{\alpha_j}z\partial_{x_i}\alpha_j$, in combination with 
\eqref{e:|DxXm1|Linf} and \eqref{e:|Da(X-I)|Linf+|Daz|Linf}.
The estimate for $\partial_t u$ in \eqref{e:|Dtxu|Linf} follows from the estimate for 
$\nabla u$ and by using Equation \eqref{e:QL1oPDE}.
\hfill$\Box$
\\

\noindent
{\bf Proof of Theorem \ref{th:u,Sigma(N,1)}}.\\
(a)
We apply Theorem \ref{th:QL1oPDE,sol} with $u_0^\pm$ instead of $u_0$,
and obtain two $C^1(\R^N\times[0,T])\cap W^{1,\infty}(\R^N\times[0,T])$ solutions $u^\pm$.
The estimates \eqref{e:|upm|Linf} and \eqref{e:|Dtxupm|Linf} follow 
from \eqref{e:|u|Linf} and \eqref{e:|Dtxu|Linf}.
\\
(b)
With $u^{\pm}$ known, we consider the problem \eqref{e:RH(N,1)}, 
with $a$ and $b$ as given in \eqref{e:ai,b}.
We note that \eqref{e:RH(N,1)} is of the form \eqref{e:QL1oPDE}, in $N-1$ dimensions instead of $N$.
Furthermore, $a$ and $b$ are $C^1$ functions because $u^-$, $u^+$ and ${g}_0$ are $C^1$. 
Then \eqref{e:RH(N,1)} has a unique solution 
$g\in C^1(\R^{N-1}\times\times[0,T])\cap W^{1,\infty}(\R^{N-1}\times[0,T])$,
which shows that 
$\Gamma=\{(x',{g}(x',t),t),\; x'\in\R^{N-1},\; t\in[0,T^*]\}$ is a $C^1$  surface.

To conclude the proof we must show that $\Gamma$ is indeed the shock wave associated with 
\eqref{e:CL(N,1)}.
Indeed, for $T$ small enough, Condition \eqref{e:shockhere} shows that on $\Gamma$ intersect 
only pairs of characteristics, one emanating from $\{{g}_0<x_N\}$ and the other from $\{{g}_0>x_N\}$, and
$\Gamma$ does not depend on the extensions $u_0^\pm$.
\hfill$\Box$

\subsubsection{LeafNet neural network approximation error estimate for scalar shock waves}\label{sss:1xN,shock,error}
In this subsection, we present an estimate of the error of the approximation of the shock wave 
$\Gamma$ by a neural network network with the {\it LeafNet} algorithm presented in Subsection 
\ref{sss:1xN,shock,alg}. 
The following is the error estimate result for the approximation of $u^\pm$.
%
\begin{theo}\label{th:|bupm-upm|Linf}
Let ${\bf u}^\pm$ be a network minimizer of $\bs{\cal L}^\pm$ in ${\bf N}_A$, see \eqref{e:argmin(cLpm[u])} i.e. 
\begin{equation}
{\bf u}^\pm={\rm argmin}\{ \bs{\cal L}^\pm[{\bf u}],\;\; {\bf u}\in{\bf N}_A\}, 
\end{equation}
and consider
$(\bs{\varepsilon}^\pm_i,\bs{\varepsilon}^\pm_b)\in 
C^\infty(\R^N\times\R)\times C^\infty(\R^N)$ defined by
\begin{equation}\label{e:varepsion(L)}
\bs{\varepsilon}^\pm_i=L^\pm_i[{\bf u}^\pm],\quad
\bs{\varepsilon}^\pm_b=L^\pm_b[{\bf u}^\pm],
\end{equation}
with $L_i^\pm$ and $L^\pm_b$ given by \eqref{e:Lpm(N,1)}. 
Denoting $e^\pm={\bf u}^\pm-u^\pm$, we have
\begin{eqnarray}
\|e^\pm\|_{L^\infty(\R^N\times[0,T])} 
&\leq &
(\|\bs{\varepsilon}^\pm_b\|_{L^\infty(\R^N)} 
+ 
T\|\bs{\varepsilon}^\pm_i\|_{L^\infty(\R^N\times[0,T])})
e^{T\|f'\|_{L^\infty(B(0,\rho^\pm_0+\bs{\rho}^\pm_0))} \rho_1^\pm},
\label{e:|bupm-upm|Linf}
\end{eqnarray}
where 
$\rho^\pm_0=\rho^\pm_0(u_0^\pm)$,
$\bs{\rho}^\pm_0=\bs{\rho}^\pm_0(u_0^\pm,\bs{\varepsilon}^\pm_i,\bs{\varepsilon}^\pm_b)$ are given by \eqref{e:|upm,bupm|Linf}
and
$\rho_1^\pm=\rho_1^\pm(\nabla u_0,f')$ is given by \eqref{e:|Dxupm|Linf}.
\end{theo}
\noindent
{\bf Proof}.
\color{black}
First, we note that  $u^\pm$ solves \eqref{e:upm(1,N)} and ${\bf u}^\pm$ solves
\eqref{e:varepsion(L)},  both problems like \eqref{e:QL1oPDE}. 
Applying \eqref{e:|u|Linf} and  \eqref{e:|Dtxu|Linf} of Theorem \ref{th:QL1oPDE,sol} to $u^\pm$ and ${\bf u}^\pm$ gives
\begin{subequations}\label{e:|upm,bupm|Linf}
\begin{align}
\|u^\pm\|_{L^\infty(\R^N\times[0,T])}
&\leq
\|u^\pm_0\|_{L^\infty(\R^N)}=:\rho^\pm_0,		
\\
\|{\bf u}^\pm\|_{L^\infty(\R^N\times[0,T])}
&\leq
\|u^\pm_0\|_{L^\infty(\R^N)}
+
\|\bs{\varepsilon}_b^\pm\|_{L^\infty(\R^N)}	
T\|\bs{\varepsilon}_i^\pm\|_{L^\infty(\R^N\times[0,T])}=:\bs{\rho}^\pm_0,  	
\end{align}
\end{subequations}
and
\begin{subequations}\label{e:|Dxupm,Dxbupm|Linf}
\begin{align}
\|\nabla u^\pm\|_{L^\infty(\R^N\times[0,T])}
&\leq
\|\nabla u_0\|_{L^\infty(\R^N)}
+
C^\pm\|f'\|_{L^\infty(B(0,\rho^\pm_0))}=:\rho^\pm_1,	\label{e:|Dxupm|Linf}
\\
\|\nabla {\bf u}^\pm\|_{L^\infty(\R^N\times[0,T])}
&\leq
\|\nabla u_0\|_{L^\infty(\R^N)}
+
\|\nabla_x \bs{\varepsilon}_b^\pm\|_{L^\infty(\R^N)} 		\nonumber\\
&+
{\bf C}^\pm
(\|f'\|_{L^\infty(B(0,\bs{\rho}^\pm_0))},
\|\nabla\bs{\varepsilon}_i^\pm\|_{L^\infty(\R^N\times[0,T])})=:\bs{\rho}^\pm_1, 	\label{e:|Dxbupm|Linf}
\end{align}
\end{subequations}
and the constants $C^\pm$, ${\bf C}^\pm$ are  as in Theorem \ref{th:QL1oPDE,sol}.

From the definition of $u^\pm$ and the assumption \eqref{e:varepsion(L)}, $e^\pm$  satisfies
\begin{subequations}\label{e:epm}
\begin{align}
\partial_t e^\pm
+
f({\bf u}^\pm)\cdot\nabla e^\pm 
&=
\bs{\varepsilon}^\pm_i(x,t) 	\nonumber\\
&
-\left(\int_0^1f'(u^\pm+s({\bf u}^\pm-u^\pm))\cdot\nabla u^\pm ds\right)e^\pm,\;\; {\rm in}\;\; \R^N\times(0,T),
\\
e^\pm(\cdot,0)&=\bs{\varepsilon}^\pm_b,\;\; {\rm on}\; \R^N.
\end{align}
\end{subequations}
We apply  estimate \eqref{e:|u|Linf} of Theorem \ref{th:QL1oPDE,sol} with
$f(z,x,t)\sim f({\bf u}^\pm(x,t))$,
$u_0\sim \bs{\varepsilon}^\pm_b$ and
$r(z,x,t)
\sim \bs{\varepsilon}^\pm_i(x,t)
-
\left(\int_0^1f'(u^\pm(x,t)+s({\bf u}^\pm(x,t)-u^\pm(x,t)))\cdot\nabla u^\pm(x,t) ds\right)z$.
Note that 
\begin{equation}\label{e:|r(0,.,.)|,thm2.3}
\|r(0,\cdot,\cdot)|_{L^\infty(\R^N\times\R)}
\leq 
\|\bs{\varepsilon}^\pm_b\|_{L^\infty(\R^N\times\R)},
\end{equation}
and by using \eqref{e:|upm,bupm|Linf} and \eqref{e:|Dxupm|Linf}  the constant $C_r$ is estimated as follows
\begin{eqnarray}
&&
\|r(z,\cdot,\cdot)-r(0,\cdot,\cdot)\|_{L^\infty(\R^N\times\R))}
\nonumber\\
&\leq&
\left\|
\int_0^1 f'(u^\pm(\cdot,\cdot)+s({\bf u}^\pm(\cdot,\cdot)-u^\pm(\cdot,\cdot)))
\cdot
\nabla u^\pm(\cdot,\cdot) ds
\right\|
|z|
\nonumber\\
&\leq&
\int_0^1
\left\|
f'((1-s)u^\pm(\cdot,\cdot)+s{\bf u}^\pm(\cdot,\cdot))
\right\|_{L^\infty(\R^N\times\R)}
\|\nabla u^\pm(\cdot,\cdot)\|_{L^\infty(\R^N\times\R)}
|z|
\nonumber\\
&\leq&
\|f'\|_{L^\infty(B(0,{\rho}_0^\pm+\bs{\rho}_0^\pm))}
\rho_1^\pm
|z|
\nonumber\\
&=:&
C_r|z|.		\label{e:C_r,thm2.3}
\end{eqnarray}
From the estimate \eqref{e:|u|Linf} with \eqref{e:|r(0,.,.)|,thm2.3} and \eqref{e:C_r,thm2.3} 
we get
\begin{eqnarray*}
\|e^\pm\|_{L^\infty(\R^N\times[0,T])} 
&\leq &
(\|\bs{\varepsilon}^\pm_b\|_{L^\infty(\R^N)} 
+ 
T\|\bs{\varepsilon}^\pm_i\|_{L^\infty(\R^N\times[0,T])}) 
e^{T\|f'\|_{L^\infty(B(0,\rho^\pm_0+\bs{\rho}^\pm_0))}\rho_1^\pm},
\end{eqnarray*}
which proves the theorem.
\color{black}
\hfill$\Box$\\

\noindent
Now we will estimate the error of approximating $g$ by the network ${\bf g}$, 
see Subsection \ref{sss:1xN,shock,alg}.
\begin{theo}\label{th:|bg-g|Linf}
Let ${\bf u}^\pm$ be a network minimizer of $\bs{\cal L}^\pm[{\bf u}]$ as in Theorem 
\ref{th:|bupm-upm|Linf} and $e^\pm={\bf u}^\pm-u^\pm$, where
\begin{equation}\label{e:varepsion(L),bis}
\bs{\varepsilon}^\pm_i=L^\pm_i[{\bf u}^\pm],\quad
\bs{\varepsilon}^\pm_b=L^\pm_b[{\bf u}^\pm].
\end{equation}
Let ${g}$ be the unique $C^1$-solution of \eqref{e:RH(N,1)} satisfying
$K[g]=(0,0)$, where $K=(K_i,K_b)$ is defined in \eqref{e:K[g]=0}.
Let also ${\bf g}$ be a network solution of \eqref{e:argmin(cK[h])},
and let $(\bs{\sigma}^\pm_i,\bs{\sigma}^\pm_b)\in C^\infty(\R^N\times[0,T])\times C^\infty(\R^N)$ 
defined by 
\begin{equation}\label{e:bfK[g]=sigma}
\bs{\sigma}_i={\bf K}_i[{\bf g}],\quad
\bs{\sigma}_b={\bf K}_b[{\bf g}],
\end{equation}
where ${\bf K}$ is defined in \eqref{e:bfK[h]}. Then 
\begin{eqnarray}
\|{\bf g}-g\|_{L^\infty(\R^{N-1}\times[0,T])}
&\leq&
\bigg(
\|\bs{\sigma}_b\|_{L^\infty(\R^{N-1})}
+
\|\bs{\sigma}_i\|_{L^\infty(\R^{N-1}\times[0,T])}
+ 		\nonumber\\
&&
\hspace*{4mm}
\sum_{*=\pm}
C^*_0
T(\|\bs{\varepsilon}^*_b\|_{L^\infty(\R^N)} 
+ 
T\|\bs{\varepsilon}^*_i\|_{L^\infty(\R^N\times[0,T])})
\bigg)
e^{
C_1T},	\label{e:|bg-g|Linf}
\end{eqnarray}
where
\begin{eqnarray}
C_0^\pm
&=&
(1+\|\nabla g\|_{L^\infty(\R^N)})\|f'\|_{L^\infty(B(0,\rho_0^++\rho_0^-+\bs{\rho}_0^++\bs{\rho}_0^-))})\cdot
\exp{\left(T\rho^\pm_1\|f'\|_{L^\infty(B(0,\rho^\pm_0+\bs{\rho}^\pm_0))}\right)},
				\label{e:Cpm0}\\
C_1&=&
(1+\|\nabla g\|_{L^\infty(\R^N)})(\rho^+_1+\rho^-_1)\|f'\|_{L^\infty(B(0,\rho_0^++\rho_0^-+\bs{\rho}_0^++\bs{\rho}_0^-))},
			\label{e:C1}\\
\rho^\pm_0&=&\|u_0^\pm\|_{L^\infty(\R^N)},\qquad (see\;\, \eqref{e:|upm,bupm|Linf})\\
\bs{\rho}^\pm_0
&=&
\|u^\pm_0\|_{L^\infty(\R^N)}
+
\|\bs{\varepsilon}_b^\pm\|_{L^\infty(\R^{N-1})}
+
T\|\bs{\varepsilon}_i^\pm\|_{L^\infty(\R^N\times[0,T])},\qquad (see\;\, \eqref{e:|upm,bupm|Linf})	\\
\rho_1^\pm&=&
\|\nabla u_0\|_{L^\infty(\R^N)}
+
C^\pm\|f'\|_{L^\infty(B(0,\rho^\pm_0))}.\qquad (see\;\, \eqref{e:|Dxupm|Linf})
\end{eqnarray}
\end{theo}
\noindent
{\bf Proof}. 
First, we establish some estimates for ${g}$ and ${\bf g}$ that will be used 
for the estimate of ${\bf g}-g$.
In view of \eqref{e:K[g]=0}, \eqref{e:bfK[h]} and \eqref{e:RH(N,1)}, 
$g$ and ${\bf g}$ solve a quasi-linear PDE identical to \eqref{e:QL1oPDE}, in dimension $N-1$ instead of $N$. Identifying for each of equations \eqref{e:K[g]=0}, \eqref{e:bfK[h]},
the functions $f$ and $r$ and $u_0$ (with reference to \eqref{e:QL1oPDE}) and 
applying the estimate \eqref{e:|u|Linf} to $g$ and ${\bf g}$ gives
\begin{eqnarray}
\hspace*{-8mm}
\|{g}\|_{L^\infty(\Omega'\times[0,T])}
&\leq&
\|{g}_0\|_{L^\infty(\R^{N-1})}
+
T\|f_N\|_{L^\infty(B(0,\rho^-_0 + \rho^+_0))}
=:
\rho,		\label{e:|g|Linf<}
\\
\hspace*{-8mm}
\|{\bf g}\|_{L^\infty(\Omega'\times[0,T])}
&\leq&
\|{g}_0\|_{L^\infty(\R^{N-1})}+\|\bs{\sigma}_b\|_{L^\infty(\R^{N-1})} \nonumber\\
&+&
T(\|f_N\|_{L^\infty(B(0,\bs{\rho}^-_0+\bs{\rho}^+_0))} + \|\bs{\sigma}_i\|_{L^\infty(\R^{N-1}\times[0,T])})
=:\bs{\rho}.		\label{e:|hg|Linf<}
\end{eqnarray}

Next, we estimate  $e:={\bf g}-g$. For this, we show first that $e$ satisfies a first order PDE like \eqref{e:QL1oPDE}.
Indeed, subtracting $K[{g}]=0$ from ${\bf K}[{\bf g}]=(\bs{\sigma}_i,\bs{\sigma}_b)$ we get
\begin{subequations}\label{e:e}
\begin{align}
\partial_t e
+
{\bf a}({\bf g},x',t)\cdot\nabla e
&=
\bs{\sigma}_i 
+
({\bf b}({\bf g},x',t)-b({g},x',t))	
-
({\bf a}({\bf g},x',t)-a({g},x',t))\cdot\nabla{g}			\nonumber\\
&=:
\bs{\sigma}_i 
+
r_b(e,x',t)- r_a(e,x',t)		\quad \mbox{\it in}\;\; \R^{N-1}\times[0,T],		\label{e:e,i}\\
e(\cdot,0)&=\bs{\sigma}_b\quad\mbox{\it on}\;\; \R^{N-1}.	\label{e:e,b}
\end{align}
\end{subequations}
We note that
the equation \eqref{e:e} is similar \eqref{e:QL1oPDE}, and $e$ is a $C^1$-solution of \eqref{e:e}.
Furthermore, we note that the fluxes ${\bf a}({\bf g},x',t)$ in \eqref{e:e,i} are $C^1$,
but the right hand side of \eqref{e:e,i} is only $C^0$ which is due to the presence of $\nabla{g}$.
However, due to the fact that ${\bf a}({\bf g},x',t)$ satisfies the assumptions \eqref{e:f->C1},
the estimate \eqref{e:|u|Linf} applies.

We will rewrite $r_a$ and $r_b$ as functions of the form $k(x',t)e$. For this set
\[
\ell(s,u,v)=su+(1-s)v,
\quad
e^\pm={\bf u}^\pm-u^\pm.
\]

Then $r_b$ can be written as follows
\begin{eqnarray}
r_b(e,x',t)
&=&
{\bf b}({\bf g}(x',t),x',t)-b({\bf g}(x',t),x',t)
+
b({\bf g}(x',t),x',t)-b(g(x',t),x',t)	\nonumber\\
&=&
\int_0^1 
(f_N(\ell(s,{\bf u}^+(x',{\bf g},t),{\bf u}^-(x',{\bf g},t)))-
 f_N(\ell(s,      u^+(x',{\bf g},t),      u^-(x',{\bf g},t))))ds	\nonumber\\
&+&
\int_0^1 
(f_N(\ell(s,u^+(x',{\bf g},t),u^-(x',{\bf g},t)))-
 f_N(\ell(s,u^+(x',      g,t),u^-(x',     g,t))))ds	\nonumber\\
&=&
\int_0^1
\int_0^1
f_N'(\ell(\alpha,\ell(s,{\bf u}^+(x',{\bf g},t),{\bf u}^-(x',{\bf g},t)),
                 \ell(s,     u^+(x',{\bf g},t),      u^-(x',{\bf g},t)))d\alpha 
                        	\nonumber\\
&&
\hspace*{5mm}
\times
\ell(s,e^+(x',{\bf g},t),e^-(x',{\bf g},t))ds			\nonumber\\
&+&
\int_0^1
\int_0^1
f_N'(\ell(\alpha,\ell(s,u^+(x',{\bf g},t),u^-(x',{\bf g},t)),
                 \ell(s,u^+(x',     g,t),u^-(x',      g,t)))d\alpha 
                        	\nonumber\\
&&
\hspace*{5mm}\times
\ell
\bigg(s, 
{\scriptsize \int_0^1}\partial_{x_N}u^+(x',\ell(\alpha,{\bf g},g),t)))d\alpha,
{\scriptsize \int_0^1}\partial_{x_N}u^-(x',\ell(\alpha,{\bf g},g),t)))d\alpha
\bigg)ds		\nonumber\\
&&
\times
e,			\label{e:rb,detailed} 
\end{eqnarray}
which shows that $r_b$ satisfies \eqref{e:r->C1} (without the term $\nabla _{(z,x)}$) with 
\begin{subequations}\label{e:|rb|Linf}
\begin{align}
\|r_b(0,\cdot,\cdot)\|_{L^\infty(\R^{N-1}\times[0,T])}
&\leq
\|f_N'\|_{L^\infty(B(0,\rho_0^++\rho_0^-+\bs{\rho}_0^++\bs{\rho}_0^-))}
\sum_{*=\pm}\|e^+\|_{L^\infty(\R^N\times[0,T])},			\label{e:|rb(0)|Linf}\\
C_r^b
&\leq
\|f_N'\|_{L^\infty(B(0,2(\rho_0^++\rho_0^-)))}
\sum_{*=\pm}\|\nabla u^*\|_{L^\infty(\R^N\times[0,T])}.
\end{align}
\end{subequations}
Similarly, for $r_a$ we have
\begin{eqnarray}
r_a(e,x',t)
&=&
\Big(
({\bf a}({\bf g}(x',t),x',t)-a({\bf g}(x',t),x',t))
+
(a({\bf g}(x',t),x',t)-a(g(x',t),x',t))
\Big)\cdot\nabla g	\nonumber\\
&=&
\int_0^1
\int_0^1
f_i'(\ell(\alpha,\ell(s,{\bf u}^+(x',{\bf g},t),{\bf u}^-(x',{\bf g},t)),
                 \ell(s,      u^+(x',{\bf g},t),      u^-(x',{\bf g},t)))d\alpha 
                        	\nonumber\\
&&
\hspace*{5mm}
\times
\ell(s,e^+(x',{\bf g},t),e^-(x',{\bf g},t))
\partial_i g ds			\nonumber\\
&+&
\int_0^1
\int_0^1
f_i'(\ell(\alpha,\ell(s,u^+(x',{\bf g},t),u^-(x',{\bf g},t)),
                 \ell(s,u^+(x',     g,t),u^-(x',      g,t)))d\alpha 
                        	\nonumber\\
&&
\hspace*{5mm}
\times
\ell
\bigg(s, 
{\scriptsize \int_0^1}\partial_{x_N}u^+(x',\ell(\alpha,{\bf g},g),t)))d\alpha,
{\scriptsize \int_0^1}\partial_{x_N}u^-(x',\ell(\alpha,{\bf g},g),t)))d\alpha
\bigg) ds	\nonumber\\
&&
\times \partial_i g 
e\, .
		\label{e:ra,detailed} 
\end{eqnarray}
Like for $r_b$, it implies that $r_a$ satisfies \eqref{e:r->C1} (without the term with $\nabla_{(z,x)}$,
which is not needed for the estimate \eqref{e:|u|Linf}) with 
\begin{subequations}\label{e:|ra|Linf}
\begin{align}
\|r_a(0,\cdot,\cdot)\|&_{L^\infty(\R^{N-1}\times[0,T])}		\nonumber\\
&\leq
\|f'\|_{L^\infty(B(0,\rho_0^++\rho_0^-+\bs{\rho}_0^++\bs{\rho}_0^-))}\|\nabla g\|_{L^\infty(\R^N)}
\sum_{*=\pm}\|e^*\|_{L^\infty(\R^N\times[0,T])},			\label{e:|ra(0)|Linf}\\
C_r^a
&\leq
\|f_N'\|_{L^\infty(B(0,2(\rho_0^++\rho_0^-)))}
\|\nabla g\|_{L^\infty(\R^N)}
\sum_{*=\pm}\|\nabla u^*\|_{L^\infty(\R^N\times[0,T])}.
\end{align}
\end{subequations}
Then applying \eqref{e:|u|Linf} to \eqref{e:e}, and using \eqref{e:|ra|Linf} and \eqref{e:|rb|Linf} 
gives
\begin{eqnarray*}
|e|
&\leq&
(\|\bs{\sigma}_b\|_{L^\infty(\R^{N-1})}
+
\|\bs{\sigma}_i\|_{L^\infty(\R^{N-1}\times[0,T])}
+ 
T\|r_a(0,\cdot)+r_b(0,\cdot)\|_{L^\infty(\R^{N-1}\times[0,T])}
)
e^{T(C_r^a+C_r^b)}		\nonumber\\
&\leq&
\big(
\|\bs{\sigma}_b\|_{L^\infty(\R^{N-1})}
+
\|\bs{\sigma}_i\|_{L^\infty(\R^{N-1}\times[0,T])}
+ 		\nonumber\\
&&
\hspace*{4mm}
T(1+\|\nabla g\|_{L^\infty(\R^N)})\|f'\|_{L^\infty(B(0,\rho_0^++\rho_0^-+\bs{\rho}_0^++\bs{\rho}_0^-))})
\sum_{*=\pm}\|e^*\|_{L^\infty(\R^N\times[0,T])}
\big)
		\nonumber\\
&&
\times
\exp
\Big(
T(1+\|\nabla g\|_{L^\infty(\R^N)})\|f'\|_{L^\infty(B(0,\rho_0^++\rho_0^-+\bs{\rho}_0^++\bs{\rho}_0^-))})
\sum_{*=\pm}\|\nabla u^*\|_{L^\infty(\R^N\times[0,T])}
\Big),		\nonumber
\end{eqnarray*}
which after replacing $\|e^\pm\|_{L^\infty(\R^N)}$ as given by \eqref{e:|bupm-upm|Linf} and 
$\|\nabla u^\pm\|_{L^\infty(\R^N)}$ as given by  \eqref{e:|Dxupm,Dxbupm|Linf}
completes the proof.
\hfill$\Box$

\begin{rem}\label{r:|bg - g|Linf}
The estimate \eqref{e:|bg-g|Linf} provides a bound for the error 
$\|{\bf g}-g\|_{L^\infty(\R^{N-1}\times[0,T])}$ in terms of data of the problem 
($u_0$ and $f$), and the network approximation residuals 
$\bs{\sigma}_i$, $\bs{\sigma}_b$, and 
$\bs{\varepsilon}_i^\pm$, $\bs{\varepsilon}_b^\pm$,
in the whole space. 
The estimate \eqref{e:|bg-g|Linf} still provides an estimate for  ${\bf g}-g$
in the case, for example $\Omega'=\prod_{i=1}^{N-1}(a_i,b_i)$ and $\Omega=\Omega'\times(a_N,b_N)$, as demonstrated by the following arguments.
\\
(i) 
We consider the extensions $u^\pm_0$ with compact support, and therefore the solution 
$u^\pm$ will have support in $D\times[0,T]$, with appropriate
$D=\prod_{i=1}^{N}(c_i,d_i)$, and $g$ will have support in $D'=\prod_{i=1}^{N-1}(c_i,d_i)$.
As  we can always arrange to have the neural networks functions identically zero outside $D\times[0,T]$,
the estimate \eqref{e:|bg-g|Linf} is reduced in $D'\times[0,T]$.
\\
(ii)
We will assume that the Sobolev index number $k,m,p,q$ in  the LeafNet algorithm in Subsection 
\ref{sss:1xN,shock,alg} are such that the embeddings 
$H^m(D)\subset L^\infty(D)$,
$H^k(D\times[0,T])\subset L^\infty(D\times[0,T])$,
$H^q(D')\subset L^\infty(D')$,
$H^p(D'\times[0,T])\subset L^\infty(D'\times[0,T])$ are continuous.
Then the estimate \eqref{e:|bg-g|Linf} provides a bound of
$\|{\bf g}-g\|_{L^\infty(D'\times[0,T])}$ of the form
\begin{eqnarray}
\|{\bf g}-g\|_{L^\infty(D'\times[0,T])}
&\leq&
\bigg(
\|\bs{\sigma}_b\|_{H^m(\R^{N-1})}
+
\|\bs{\sigma}_i\|_{H^k(D'\times[0,T])}
+ 		\nonumber\\
&&
\hspace*{4mm}
\sum_{*=\pm}
D^*_0
T(\|\bs{\varepsilon}^*_b\|_{H^q(D)} 
+ 
T\|\bs{\varepsilon}^*_i\|_{H^p(D\times[0,T])})
\bigg)
e^{D_1T},	\label{e:|bg-g|Linf(D'x[0,T]}
\end{eqnarray}
with appropriate constants $D^*_0$, $D_1$.
The estimate \eqref{e:|bg-g|Linf(D'x[0,T]} states that the $L^\infty$ error of shock approximation
is bounded by the Sobolev norm of network residuals 
$\bs{\sigma}_i$, $\bs{\sigma}_b$, and 
$\bs{\varepsilon}_i^\pm$, $\bs{\varepsilon}_b^\pm$.
\end{rem}
%

\subsection{Shock wave formation}\label{ss:1xN,shock,form}
The method presented in Section \ref{ss:1xN,shock,method}  can be used to handle the shock formation.
Namely, we assume that at the initial time there is no shocks, and at a certain time 
$t_*\in(0,T)$ a shock appears, which develops  until the time $T$.
We assume that there is only one shock which is formed, and we denoted by $\Gamma$ the corresponding discontinuity surface. 

Let $t_1\in(t_*,T)$ and set $\Gamma_1=\Gamma\cap\{t=t_1\}$. We assume that $\Gamma_1$ is the graph of a $C^1$-function $g_1$, i.e. $\Gamma_1=\{(x',{g}_1(x')),\; x'\in\Omega'\}$. 
If $u_1(x):=u(x,t_1)$ were known, then we could proceed as in Section \ref{ss:1xN,shock,method} to track the shock,
with $u_1$ instead of $u_0$ and $\Gamma_1$ instead of $\Gamma_0$.

Even though $u_1$ and $\Gamma_1$ are not known, we can still proceed as in Subsection 
\ref{ss:1xN,shock,method} as follows. 
We consider a virtual discontinuity curve $\Gamma_0$ with equations $x_N=g_0(x')$, which separates 
$\Omega$ in $\Omega^\pm$, and generates two initial functions $u_0^\pm$ as in \eqref{e:u0=u0+,u0-}.
Of course $u_0^-=u_0^+$ on $\Gamma_0$.
Next, we solve $u^\pm$ and $g$ as in \eqref{e:upm(1,N)} and \eqref{e:RH(N,1)}.
The time formation $t_*$ is smallest time when $\jump{u}$ on $\Gamma$ becomes discontinuous.
Therefore, the solution $u(\cdot,t)$ is continuous for $t\in[0,t_*]$, 
any $g$ will solve \eqref{e:RH(N,1)}.
However, for $t\in(t_*,T)$, ${g}$ must satisfy the Rankine-Hugoniot condition and therefore 
${g}$ will parameterize the shock surface $\Gamma$.  
We illustrate this method with numerical results in Section \ref{s:numerics}.

\subsection{Neural network solution for rarefaction waves}\label{ss:1xN,rw}
In this section we assume that $u_0$ has only one discontinuity 
$\Gamma_0=\{(x',{g}_0(x'),\; x'\in\Omega'\}$, which develops a rarefaction wave. 
We can  apply the classical PINN methods, which provide a smooth approximation of $u$ with neural networks. 
As $u$ is discontinuous along $\Gamma_0$, this network will then not approximate accurately the solution near $\Gamma_0$.

\subsubsection{Discontinuous rarefaction waves}\label{subsubsec:rw2}
We can approximate sharply the solution even near the discontinuity curve $\Gamma_0$ by using a kind of 
self similar solution as follows.
As the discontinuity on $\Gamma_0$ develops a rarefaction wave, there exists a domain
$Q_{\textrm{rw}}\subset Q$, bounded by the characteristic surfaces $\Gamma^\pm$, 
made of characteristic curves starting on $\Gamma_0$ given by
\begin{eqnarray}
\hspace*{-10mm}
\Gamma^\pm
&=&
\{(x',x_N,t),\;\; 
x_i = \alpha_i + t f_i(u_0^\pm(\alpha,{g}_0(\alpha))),\;\;  (\alpha,t)\in Q',\;\; i=1,\ldots,N-1,
\nonumber\\
\hspace*{-10mm}
&&
\hspace*{23mm}
x_N = {g}_0(\alpha) + t f_N(u_0^\pm(\alpha,{g}_0(\alpha)) \} \, .		\label{e:Sigmapm}
\end{eqnarray}
At least for small $T$, the map $(\alpha,t)\in Q'\mapsto (x',t)$, $x'=(x_1,\ldots,x_{N-1})$,
is invertible. Then \eqref{e:Sigmapm} provides $x_N=g^\pm(x',t)$.
Depending on $f$ and $u_0^\pm$, one may solve $g^\pm$ analytically.
Alternatively, ${g}^\pm$ can be evaluated numerically, see Subsection \ref{sssec:alg,rw}.
In the following, we assume that $g^\pm$ are known, and therefore
$\Gamma^\pm$ is the graph of the function $x_N={g}^\pm(x',t)$.
Hence, we have
\begin{equation}
Q_{\textrm{rw}}=\{(x',x_N,t),\;\; x_N\in({g}^-(x',t),{g}^+(x',t)),\; (x',t)\in Q'\}.
\end{equation}
The solution $u$ satisfies
\begin{subequations}\label{e:u,rw}
\begin{align}
\partial_t u + f(u)\cdot\nabla u & = 0,\;\;\; \mbox{in}\;\; Q_{\textrm{rw}},
\\
u&=u^-\;\; \mbox{on}\;\; \Gamma^-,
\\
u&=u^+\;\; \mbox{on}\;\; \Gamma^+.
\end{align}
\end{subequations}
As the solution $u$ is discontinuous on $\Gamma_0$, to approximate it accurately with neural networks,
we consider the change of variables
\[
(x',x_N,t)\in Q\mapsto 
(x',s_N,t)\in \Omega'\times(0,1)\times(0,T),\quad
s_N=\cfrac{x_N-{g}^-(x',t)}{{g}^+(x',t)-{g}^-(x',t)},
\]
with corresponding inverse
\[
(x',s_N,t)\mapsto
(x',x_N,t),\quad
x_N={g}^-(x',t) + s_N({g}^+(x',t)-{g}^-(x',t)),
\]
and define $w$ by 
\begin{equation}\label{e:w,rw}
w(x',s_N,t)=u(x',x_N,t).
\end{equation}
We can write \eqref{e:u,rw} in terms of the function $w$. For this we evaluate 
\begin{eqnarray*}
\partial_t u(x',x_N,t)
&=& 
\partial_t w(x',s_N,t) 
\\
&-&
\partial_{s_N} w(x',s_N,t) 
\bigg(
\cfrac{\partial_t{g}^-(x',t)}{{g}^+(x',t)-{g}^-(x',t)} +
\\
&&
\hspace*{29mm}
\cfrac{x_N-{g}^-(x',t)}{({g}^+(x',t)-{g}^-(x',t))^2}
\partial_t ({g}^+(x',t)-{g}^-(x',t))
\bigg) \, ,
\\
\partial_{x_i} u(x',x_N,t)
&=&
\partial_{x_i}w(x',s_N,t) 
\\
&-&
\partial_{s_N} w(x',s_N,t) 
\bigg(
\cfrac{\partial_{x_i}{g}^-(x',t)}{{g}^+(x',t)-{g}^-(x',t)}+
\\
&&
\hspace*{29mm}
\cfrac{x_N-{g}^-(x',t)}{({g}^+(x',t)-{g}^-(x',t))^2}
\partial_{x_i}({g}^+(x',t)-{g}^-(x',t))
\bigg),
\\
\partial_{x_N} u(x',x_N,t)
&=&
\partial_{s_N} w(x',s_N,t) \cfrac{1}{{g}^+(x',t)-{g}^-(x',t)},
\end{eqnarray*}
for $i=1,\ldots,N-1$.
Combining these relations with \eqref{e:u,rw} gives for $w=w(x',s_N,t)$:
\begin{subequations}\label{e:Lw,rw}
\begin{align}
\partial_tw
+ \sum_{i=1}^{N-1} h_i\partial_{x_i}w 
+ h_N\partial_{s_N}w 
&=0,\;\;
\mbox{in}\;\; \{(x',s_N,t)\in \Omega'\times(0,1)\times(0,T)\},
\\
w(x',0,t)&=u^-(x',{g}^-(x',t),t),\;\; \mbox{on}\;\; \Omega'\times (0,T),
\\
w(x',1,t)&=u^+(x',{g}^+(x',t),t),\;\; \mbox{on}\;\; \Omega'\times(0,T),
\\
w(x',s_N,0)&=w_0(x',s_N),\;\; \mbox{on}\;\; \Omega'\times(0,1),
\end{align}
\end{subequations}
with $w_0$ any smooth function satisfying 
\begin{equation}
w_0(x',0)=u^-(x',g_0(x')),\quad
w_0(x',1)=u^+(x',g_0(x')),\quad x'\in\Omega',			\label{e:w0}
\end{equation}
for example
\[
w_0(x',s_N)
=(1-s_N)u_0^-(x',{g}_0(x')) 
+
s_Nu_0^+(x',{g}_0(x')),
\] 
and
\begin{eqnarray}
h_i(w,x',s_N,t)
&=&
f_i(w),				\label{e:hi(...)}
\\
h_N(w,x',s_N,t)
&=&
-
\bigg(
\cfrac{\partial_t{g}^-(x',t)}{{g}^+(x',t)-{g}^-(x',t)} 
+		\nonumber\\
&&
\hspace*{7mm}
\cfrac{x_N(x',s_N,t)-{g}^-(x',t)}{({g}^+(x',t)-{g}^-(x',t))^2}
\partial_t ({g}^+(x',t)-{g}^-(x',t))
\bigg)
\nonumber\\
&&
-f_i(w)
\bigg(
\cfrac{\partial_{x_i}{g}^-(x',t)}{{g}^+(x',t)-{g}^-(x',t)}+
\nonumber\\
&&
\hspace*{16mm}
\cfrac{x_N(x',s_N,t)-{g}^-(x',t)}{({g}^+(x',t)-{g}^-(x',t))^2}
\partial_{x_i}({g}^+(x',t)-{g}^-(x',t))
\bigg)
\nonumber\\
&&
-
f_N(w)
\cfrac{1}{{g}^+(x',t)-{g}^-(x',t)}.						\label{e:hN(...)}
\end{eqnarray}
We can solve $w$ in \eqref{e:Lw,rw} and then the solution $u$ of \eqref{e:u,rw} is given by
\begin{equation}\label{e:u,rw,up,w,um}
u(x',x_N,t)
=
\left\{
\begin{array}{ll}
u^-(x',x_N,t),& x_N<{g}^-(x',t),
\\
w\left(x',\cfrac{x_N-{g}^-(x',t)}{{g}^+(x',t)-{g}^-(x',t)},t\right),&
	{g}^-(x',t)<x_N<{g}^+(x',t),
	\\
u^+(x',x_N,t),& x_N>{g}^+(x',t) \, .
\end{array}
\right.
\end{equation}

\subsubsection{Neural network algorithm for rarefaction waves}\label{sssec:alg,rw}
We implement the method presented in Subsection \ref{ss:1xN,rw} with neural networks. The method yet again benefits from the accurate approximation of rarefaction waves avoiding spurious regularization. More specifically, we introduce two  $N+1$-dimensional neural networks ${\bf u}^\pm(x',x_N,t)$ for solving \eqref{e:upm(1,N)}.
In addition we introduce two $N$-dimensional neural networks 
${\bf g}^\pm$ for describing the curved faces $\Gamma^\pm$ of $Q_{\textrm{rw}}$ as defined in \eqref{e:Sigmapm},
and one $N+1$-dimensional neural network ${\bf w}(x',s_N,t)$ for solving \eqref{e:Lw,rw}.
Then we proceed as follows

\begin{enumerate}
\item 
\mbox{\it
Find approximate solutions ${\bf u}^\pm$ to IVP \eqref{e:upm(1,N)}, by solving \eqref{e:argmin(cLpm[u])}}.
\item 
{\mbox\it
If possible, solve analytically  $x_N=g^\pm(x',t)$ as given by 
\eqref{e:Sigmapm}, i.e.} 
\begin{subequations}\label{e:alphapm}
\begin{align}
x_i &= \alpha_i + t f_i(u_0^\pm(\alpha,{g}_0(\alpha))),\quad  i=1,\ldots,N-1,
 						\quad \alpha=(\alpha_1,\ldots,\alpha_{N-1}),
\\
x_N &= g_0(\alpha) + t f_N(u_0^\pm(\alpha,{g}_0(\alpha))).
\end{align}
\end{subequations}
\mbox{\it If not, we search for a neural network approximation of $x_N=g^\pm(x',t)$}. 
In this goal, we first introduce the following operator
\begin{eqnarray}
G^\pm[g]
&=&
g_0(\alpha) + t f_N(u_0^\pm(\alpha,{g}_0(\alpha)))
- 
g\Big(
\prod_{i=1}^{N-1}(\alpha_i + t f_i(u_0^\pm(\alpha,{g}_0(\alpha)))),t
\Big). 			\label{e:Hpm[h]}
\end{eqnarray}

Then, given an architecture $A$ in $N$ dimensions, with parameters $\Theta_A$ and corresponding set of admissible functions
${\bf N}_A$, and for $k\in\N_0$ and ${\bf g}\in{\bf N}_A$ we define
\begin{subequations}
\begin{align}
\bs{\cal G}^\pm[{\bf g}]
&=      
\sum_{\substack{\alpha\in\N_0^{N-1},\ \beta\in\N_0\\ |\alpha|+\beta\leq k}}
\int_0^T
\int_{\Omega'}
|D^\alpha_{x'}D^\beta_t G^\pm_i[{\bf g}]|^2 dx' dt \, .			\label{e:cHpm(N,1)}
\end{align}
\end{subequations}
We then look for ${\bf g}^\pm$ defined by 
\begin{equation}
{\bf g}^\pm=
\textrm{argmin}\{\bs{\cal G}^\pm[{\bf g}],\;\; {\bf g}\in {\bf N}_A\},	\label{e:min(cHpm(N,1))}
\end{equation}
hence providing the network approximations ${\bf g}^\pm$ of $g^\pm$.
In the following we assume that $g^\pm$ or ${\bf g}^\pm$ are known, and for simplicity they are denoted by $g^\pm$.
\item
\mbox{\it Solution to the IVP \eqref{e:Lw,rw}}.
Motivated by \eqref{e:Lw,rw} we introduce the operator $R$ as follows
\begin{subequations}\label{e:L[w]}
\begin{align}
R[w]&=(R_i[w],R_\pm[w],R_b[w]),\\
R_i[w]&
=\partial_tw
+ \sum_{i=1}^{N-1} h_i\partial_{x_i}w + h_N\partial_{s_N}w,\\
R_\pm[w]&=
(w(x',0,t)-u^-(x',{g}^-(x',t),t),
w(x',1,t)-u^+(x',{g}^+(x',t),t)),
\\
R_b[w]&=w(x',s_N,0)-w_0(x',s_N),
\end{align}
\end{subequations}
and for a given architecture $A$ in $N+1$ dimensions, with feasible network space
${\bf N}_A$, consider the operator $\bs{\cal R}[{\bf v}]$, ${\bf v}\in{\bf N}_A$, given by
\begin{subequations}\label{e:cL[w]}
\begin{align}
\bs{\cal R}[{\bf v}]
= &      
\sum_{\substack{\alpha\in\N_0^N,\beta\in\N_0\\ |\alpha|+\beta\leq k}}
\int_0^T
\int_0^1
\int_{\Omega'}
|D^{\alpha'}_{x'}D^{\alpha_N}_{s_N}D^\beta_t R_i[{\bf v}]|^2 dx'ds_Ndt			\label{e:cLi[w]}
\\
&+
\sum_{\substack{\alpha'\in\N_0^{N-1},\beta\in\N_0\\ |\alpha'|+\beta\leq m}}
\int_0^T
\int_{\Omega'}
|D^{\alpha'}_{x'}D^\beta_t R_\pm[{\bf v}]|^2 dx'dt			\label{e:cLb1[w]}
\\
&+
\sum_{\substack{\alpha\in\N_0^N\\ |\alpha|\leq m}}
\int_0^1\int_{\Omega'}
|D^{\alpha'}_{x'}D^{\alpha_N}_{s_N}R_b^2[{\bf v}]|^2 dx'ds_N,		\label{e:cLb2[w]}
\end{align}
\end{subequations}
where $\alpha=(\alpha',\alpha_N)$.
Then ${\bf w}\in{\bf N}_A$ is defined by
\begin{equation}
  	{\bf w} = \textrm{argmin}\{\bs{\cal R}[{\bf v}],\;\; {\bf v}\in{\bf N}_A\}. 	  \label{e:min(L[w])}
\end{equation}
\item
\mbox{\it 
Reconstruction of the approximate solution ${\bf u}$ to \eqref{e:CL(N,1)}, for $(x',x_N,t)\in Q$}
    \begin{eqnarray*}
      {\bf u}(x',x_N,t)=
      \left\{
      \begin{array}{ll}
        {\bf u}^-(x',x_N,,t), & \textrm{ if } x_N<{ g}^-(x',t)\, ,
        \vspace*{2mm}\\
	    \displaystyle{{\bf w}\left(x',\frac{x_N-{g}^-(x',t)}{{ g}^+(x',t)-{ g}^-(x',t)},t\right)}, 
	    &\textrm{ if }  x_N \in({ g}^-(x',t),{ g}^+(x',t)), 
        \vspace*{2mm}\\
        {\bf u}^+(x',x_N,t), & \textrm{ if } x_N>{ g}^+(x',t)\, .       
        \end{array}
      \right.
    \end{eqnarray*}
\end{enumerate}

\subsection{Systems of hyperbolic conservation laws in $N$ dimensions}\label{s:mxN,HCL}
Here we provide a short discussion on the extension of the methods we developed in Section \ref{s:1xN,HCL}
for hyperbolic scalar conservation laws to $N$ dimensional hyperbolic systems of conservation laws, given by
\begin{subequations}\label{e:CL(N,m)}
\begin{align}
    \partial_t u + \nabla\cdot F(u) 
    :&=
    \partial_t u + \sum_{k=1}^N\partial_k(F_k(u)) \nonumber\\
    & = 
    \partial_t u + \sum_{k=1}^N\nabla F_k(u)\cdot \partial_k u
    = 0\hspace*{10mm} \textrm{ in } Q=\Omega\times(0,T), \\
    u(\cdot,0) & =  u^0(\cdot),\hspace*{5mm} \textrm{ on } \Omega,
\end{align}
\end{subequations}
where 
$\Omega\subset\R^N$ is open, 
$u: (x,t)=(x_1,\ldots,x_N,t)\in Q \mapsto u(x,t)=(u_i(x,t)) \in \R^m$, $m\geq2$,
$\partial_t u =(\partial_t u_i)\in\R^m$,
$\partial_k u =(\partial_k u_i)\in\R^m$, 
$F(u)=(F_1(u),\ldots,F_N(u)):\R^m\mapsto \R^{m\times N}$, with 
$F_k=(F_k^i)\in\R^m$,
$\nabla F_k=(\partial_j F_k^i)\in\R^{m\times N}$,
$F_k^i \in C^2(\R;\R)$,
$i,j=1,\ldots,m$, $k=1,\ldots,N$.

Moreover, we assume that $u^0\in L^1(\Omega;\R^m)\cap TV(\Omega;\R^m)$. We again impose incoming flux boundary condition $\partial \Omega$. We refer to \cite{jmg} for a full description of the incoming flux conditions for systems.

\subsubsection{Shock waves for the systems of hyperbolic conservation laws}\label{ss:mxN,shock}
The LeafNet algorithm we have presented in Subsection \ref{ss:1xN,shock,method} can be applied for systems, with  minimal changes. Indeed, the Rankine-Hugoniot condition \eqref{e:RH(N,1)}, which is the key ingredient for tracking the shock wave, in the case of systems is written  componentwise as
\begin{eqnarray}
\nu_t\jump{u_i} + \sum_{k=1}^N\jump{F_k^i(u)}\nu_k
&=& 
0,\quad i=1,\ldots,m,			\label{e:RH(N,m)-1}
\end{eqnarray}
where $\nu$ is given in \eqref{e:nu(N)}. 
As in Subsection \ref{ss:1xN,shock,method}, rearranging \eqref{e:RH(N,m)-1}, we deduce that the function
$x_N={g}(x',t)$ representing the discontinuity surface $\Gamma$ satisfies 
\begin{subequations}\label{e:RH(N,m)}
\begin{align}
\partial_t{g} 
+ 
\frac{\jump{F_1^i(u)}}{\jump{u_i}}\partial_1{g}
+\cdots+
\frac{\jump{F_{N-1}^i(u)}}{\jump{u_i}}\partial_{N-1}{g}
&=:					\nonumber\\
\partial_t{g} 
+
a_1^i(g,x',t)\partial_1g
+\cdots+
a_{N-1}^i(g,x',t)\partial_{N-1}g
&=
r^i(g,x',t)				  \label{e:RH(N,m),i}\\
&:=
\frac{\jump{F_N^i(u)}}{\jump{u_i}},\quad i=1,\ldots,m,	\nonumber\\
{g}(\cdot,0)
  &=
  {g}_0(\cdot) \, .	 	 \label{e:RH(N,m),b}
\end{align}
\end{subequations}
We assume that $F$ is $C^2$ and therefore, each equation of \eqref{e:RH(N,m)} has a unique solution,  which is $g$. Therefore, we can find $g$ by solving any of the equations \eqref{e:RH(N,m)}. Then, theoretically, it would be enough to solve just one of the equations \eqref{e:RH(N,m)}. We note that from numerical viewpoint, it may not be appropriate because the equation that one decides to solve may represent instabilities, which  may happen if the corresponding  $\jump{u_i}$ is small, and may result in a poor solution of the other equations. In our numerical simulations, we have considered the following equivalent form of \eqref{e:RH(N,m)}
\begin{subequations}\label{e:RH(N,m2)}
\begin{align}
{\jump{u_i}}\partial_t{g} 
+ 
\jump{F_1^i(u)}\partial_1{g}
+\cdots+
\jump{F_{N-1}^i(u)}\partial_{N-1}{g}
&=
\jump{F_N^i(u)},\quad i=1,\ldots,m,\label{e:RH(N,m),i2}\\
{g}(\cdot,0)
  &=
  {g}_0(\cdot) \, .	 	 \label{e:RH(N,m),b2}
\end{align}
\end{subequations}
Hence, our algorithm for constructing $g$ consists in solving each equation of \eqref{e:RH(N,m2)} by minimizing the total loss, which is the sum of losses corresponding to each equation of \eqref{e:RH(N,m2)}. Hence, ``solving'' the shock $\Gamma$ of \eqref{e:CL(N,m)} is equivalent to solve the $N-1$ dimensional 
IVP \eqref{e:RH(N,m2)}. The LeafNet method consists in:
%
\begin{equation}\label{m:shock-wave,Nxm}
\left.
\begin{array}{lp{140mm}}
1)& 
\mbox{\it Solving (independently) the two IBVP like \eqref{e:upm(1,N)},  associated to \eqref{e:CL(N,m)},
which provide }\\
&
\mbox{\it $u^-$ and $u^+$ in $Q$}.\\
2)&
\mbox{\it Solving the Rankine-Hugoniot \eqref{e:RH(N,m2)}, which provides ${g}$ in $Q'$}.
\\
3)& 
\mbox{\it Reconstructing the solution from the initial IBVP \eqref{e:upm(1,N)} for $(x',x_N,t)\in Q$,}
\\
& \mbox{\it as follows}
    \begin{eqnarray*}
      u(x',x_N,t)=
      \left\{
      \begin{array}{ll}
        u^-(x',x_N,t), & \textrm{ if } x_N<{g}(x',t)\, ,\\
         u^+(x',x_N,t), & \textrm{ if } x_N>{g}(x',t)\, .       
        \end{array}
      \right.
      \end{eqnarray*}
\end{array}
\right.
\end{equation}
The neural network implementation of this algorithm is made as in Subsection 
\ref{sss:1xN,shock,alg},  but with vector valued neural networks.

\subsubsection{Shock wave formation}\label{ss:mxN,wave-formation}\label{subsec:shock-form,Nxm}
We can apply also the method we presented in Subsection \ref{ss:1xN,shock,form} for finding 
the time $t_*$ of the shock wave formation and tracking the shock wave.

Indeed, like in Subsection  \ref{ss:1xN,shock,form} for tracking the scalar HCL.
We start with a ``virtual" discontinuity $\Gamma_0$.
Then we approximate $u^\pm$, and next the solution $g$ to Rankine-Hugoniot equation \eqref{e:RH(N,m)}. The time formation $t_*$ is smallest time when $\jump{u}$ on $\Gamma$ becomes discontinuous.

\section{Numerical experiments}\label{s:numerics}
In this section, we present a series of numerical experiments illustrating the methods that we have developed in this paper, and applied to two-dimensional scalar shock waves, one- and two-dimensional rarefaction waves, and two-dimensional systems of HCL.  \\
In all the numerical computations, we use fully connected neural networks with $\tanh$-activation functions, and the loss functions are constructed with $H^1$-norms. A standard gradient descent algorithm from the library {\tt optax} is used in the training of the networks. In addition, the initial and boundary conditions are directly encoded in the neural networks. Finally, in most of the computations the learning rate is taken equal to $2\times 10^{-3}$. 
\subsection{Scalar HCL}
In this subsection, we propose a numerical illustration of the above strategy applied to scalar HCL 
for tracking shock waves, rarefaction waves and shock formation, see subsections \ref{ss:1xN,shock,method}, \ref{ss:1xN,rw} and \ref{ss:1xN,shock,form}. 

\medskip

\noindent{\bf Experiment 1 (2-D shock)} 
Here we consider \eqref{e:CL(N,1)}, with $F_1(u)=u^4/4$, $F_2(u)=u^2/4$, $\Omega=(a_1,b_1)\times(a_2,b_2)$,  and the following initial condition $u_0$
\begin{eqnarray*}
  u_0(x_1,x_2) =
  \left\{
  \begin{array}{ll}
    {u}_0^l(x_1,x_2), & x_1< {g}_0(x_2) \, , \\
    {u}_0^r(x_1,x_2), & x_1> {g}_0(x_2) \,  ,
    \end{array}
  \right.
\end{eqnarray*}
where ${g}_0(x_2)=\beta (x_2-a_2)^2$ and
\begin{eqnarray*}
  \left.
  \begin{array}{lcl}
  {u}_0^l(x_1,x_2)     & = &     \gamma (x_1+b_1)\big(1-\tanh[\epsilon({g}_0(x_2)-(x_1-(a_1+b_1)/2))]\big)\\
& &  +  \gamma ({g}_0(x_2)+b)\big(1-\tanh[\epsilon({g}_0(x_2)-(x_1-(a_1+b_1)/2))]\big),\\
 
  u_0^r(x_1,x_2) &  = & 1 \, .
\end{array}
\right.
\end{eqnarray*}
and $\epsilon$, $\beta$, $\gamma$ positive constants. Numerically we select $a_2=b_2=-a_1=-b_1=2$, $\epsilon=20$, $\beta=1/20$, $\gamma=1/4$, and $T=1$.  In the numerical experiments, the networks for approximating the IBVP on $u^{l,r}$ have 2 hidden layers with 5 neurons, and have $1$ hidden layer with 10 neurons for ${g}$.  The networks are trained with $300$ randomly chosen collocation points in $x_1$, $x_2$ and $t$. We report in Fig. \ref{fig1a}, the graph of $u^l_0$ (Left) after extension, and the graph of $u_0$ (Right) in $\Omega$. We then report in Fig. \ref{fig1b}, the graph of $u^l(\cdot,\cdot,T)$ (Left) as well as the graph of the reconstructed solution $u(\cdot,\cdot,T)$ (Right). Finally, we report in Fig. \ref{fig1c}, the graph of ${g}$ in 
$Q_2$ defining the time-space surface of discontinuity.
\begin{figure}[hbt!]
\begin{center}
  \includegraphics[height=6cm,keepaspectratio]{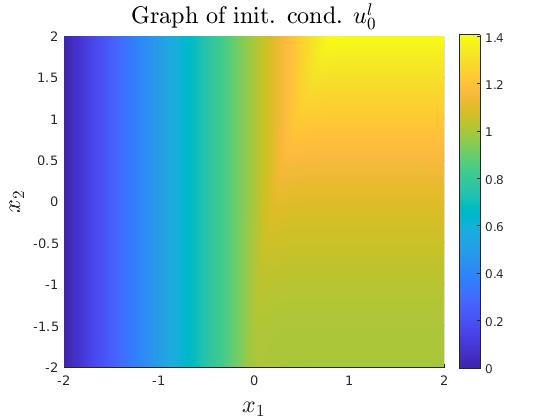}
  \includegraphics[height=6cm,keepaspectratio]{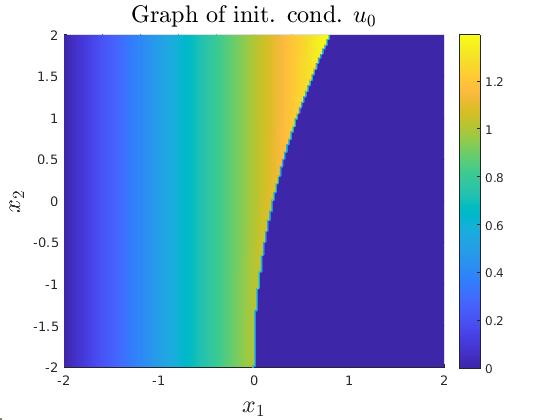}
\end{center}
\caption{{\bf Experiment 1.} (Left) Graph of  $u^l_0$. (Right) Graph of $u_0$.}
\label{fig1a}
\end{figure}

\begin{figure}[hbt!]
\begin{center}
  \includegraphics[height=6cm,keepaspectratio]{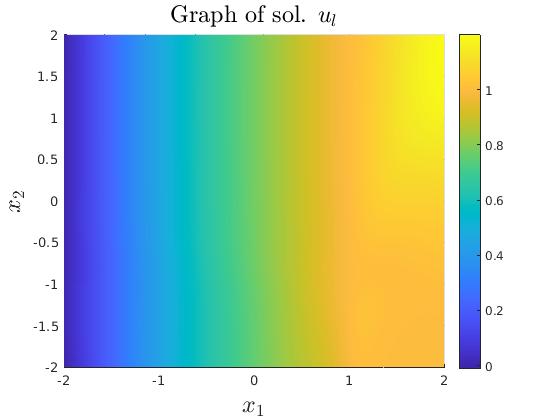}
  \includegraphics[height=6cm,keepaspectratio]{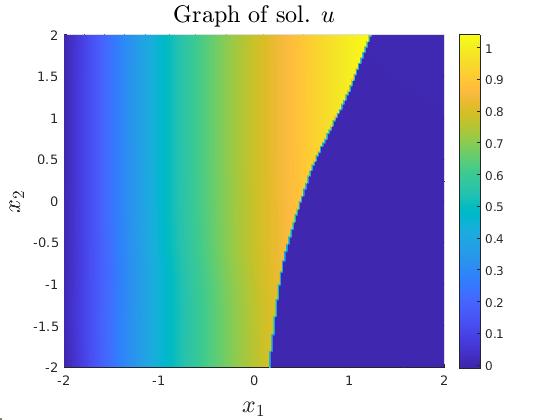}
\end{center}
\caption{{\bf Experiment 1.} 
(Left) Solution ${\bf u}_l(\cdot,\cdot,T)$. 
(Right) Solution ${\bf u}(\cdot,\cdot,T)$.}
\label{fig1b}
\end{figure}
For completeness, we also report the graph of convergence of the minimization algorithm in Fig. \ref{fig1c} (Right).
\begin{figure}[hbt!]
\begin{center}
  \includegraphics[height=6cm,keepaspectratio]{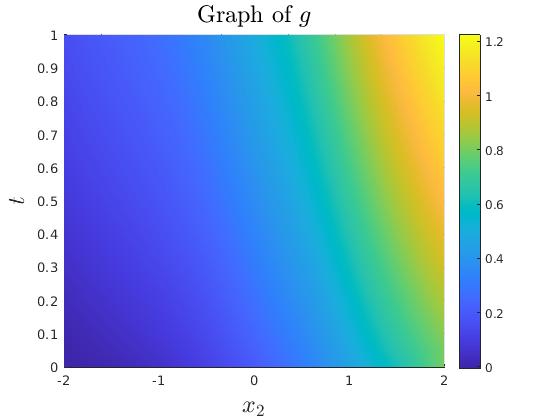}
   \includegraphics[height=6cm,keepaspectratio]{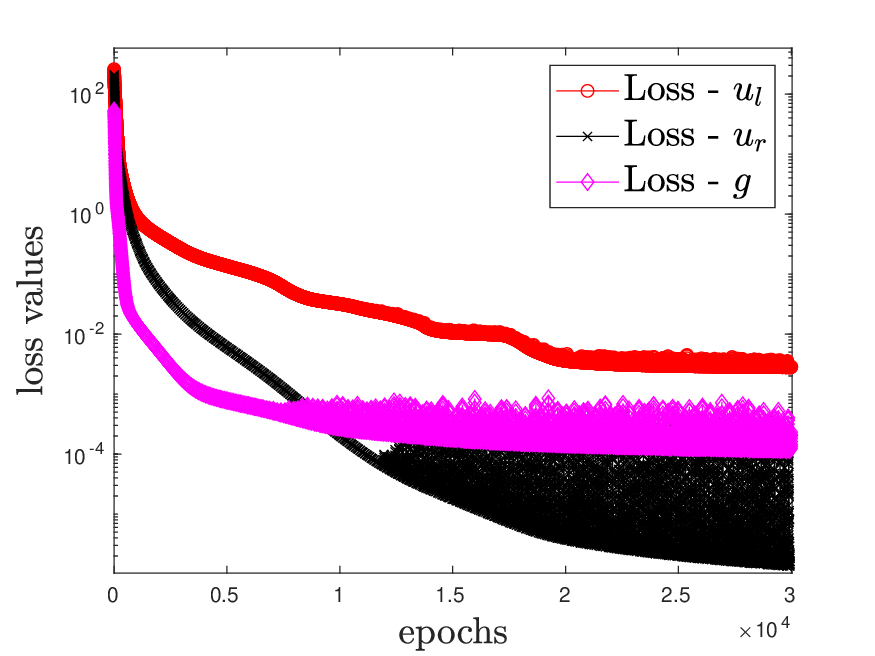} 
\end{center}
\caption{{\bf Experiment 1.} (Left) Graph of ${\bf g}$. (Right) Loss functions for training 
${\bf u}_l$, ${\bf u}_r$ and ${\bf g}$.}
\label{fig1c}
\end{figure}
This first test illustrates the zero-diffusion approximation of the two-dimensional shock wave.  \color{black} Fig. \ref{fig2a} highlights several computational advantages of the LeafNet approach for the approximation of shock waves. The LeafNet algorithm provides a sharp representation of the shock discontinuity, whereas standard computational methods would exhibits a noticeable smoothing of the shock. The loss curves reported in Fig. \ref{fig2b} indicate favorable convergence behavior. Another important computational advantage is the embarrassingly parallel nature of the method, since the computations associated with the different local space-time solutions can be carried out independently, making the approach naturally suited to parallel architectures. Overall, this result indicates that exploiting the intrinsic mathematical structure of the solution facilitates the accurate approximation of shock waves, preserves sharp discontinuities without introducing artificial numerical diffusion, and provides a highly parallelizable computational framework.
\color{black}

\medskip

\noindent{\bf Experiment 2 (1-D rarefaction wave)} 
This experiment is devoted to the approximation of one-dimensional rarefaction waves for the following Riemann problem on $Q=\Omega\times [0,T]$, where $\Omega=(-3/2,3/2)\times(-2,2)$ and $T=1.0$.
\begin{eqnarray*}
  u_0(x_1,x_2)=\left\{
  \begin{array}{ll}
    u_l, & x_1< 0,\\
    u_r, & x_1> 0 \, , 
  \end{array}
  \right.
\end{eqnarray*}
with $u_l=0$ and $u_r=1$, and the flux function $F=(F_1,F_2)$ is defined by $F_1(u)=u^2/2$, $F_2(u)=u^2/8$. We implement the method proposed in Subsection \ref{ss:1xN,rw} with $1$ layer and $10$ neurons in a one-dimensional setting. The networks are trained with $300$ randomly chosen collocation points for each variable.

Following the strategy described in Subsection \ref{subsubsec:rw2}, we simply get (as $u^l=0$, $u^r=1$ and $g_0(x_2)=0$) $g^l(x_2,t) = 0$, $g^r(x_2,t)=t$ and $s_1=x_1/t$. Hence we recover the standard one-dimensional (as the solution is here independent of $x_2$ and $t$) rarefaction~wave~structure~$w(s_1,x_2)$. 

We then search for neural network ${\bf w}$ solution approximate $w$ as proposed in Subsubsection \ref{sssec:alg,rw}. We report in Fig. \ref{fig2a} the solution ${\bf u}$ at $t=1/4$ (Top-Left) and at $t=T=1$ (Top-Right).  Similarly, we report the corresponding  solutions obtained with a direct PINN method and at $t=1/4$ (Bottom-Left) and $t=1$ (Bottom-Right). In order to get a fair comparison, we use the same number of layers and neurons in both methods. This example illustrates that the use of self-similar neural networks allows for an accurate approximation of rarefaction waves with artificial regularization of non-smooth regions.
\begin{figure}[hbt!]
\begin{center}
  \includegraphics[height=5.5cm,keepaspectratio]{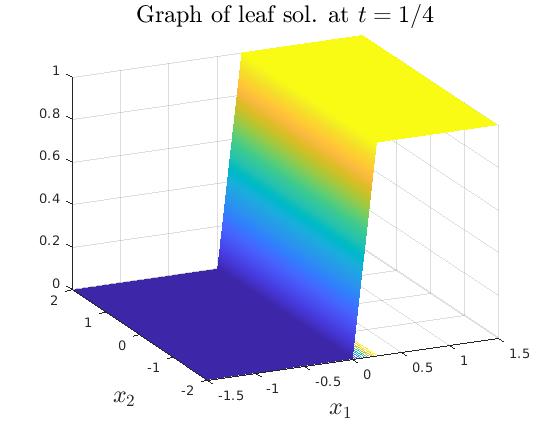}
  \includegraphics[height=5.5cm,keepaspectratio]{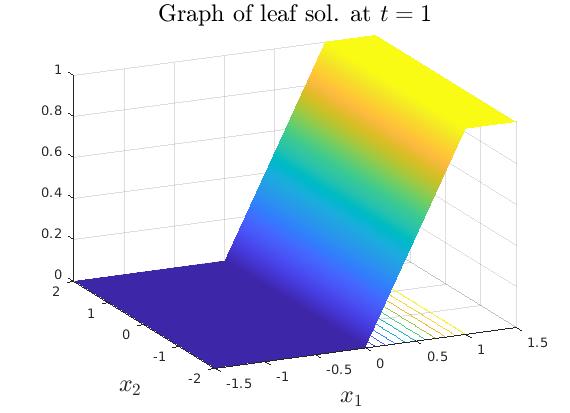}\\
    \includegraphics[height=5.5cm,keepaspectratio]{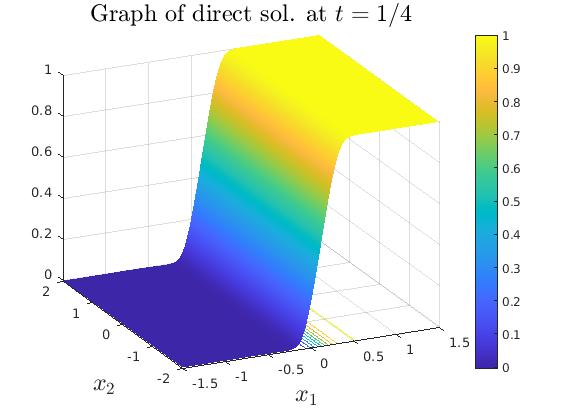}
  \includegraphics[height=5.5cm,keepaspectratio]{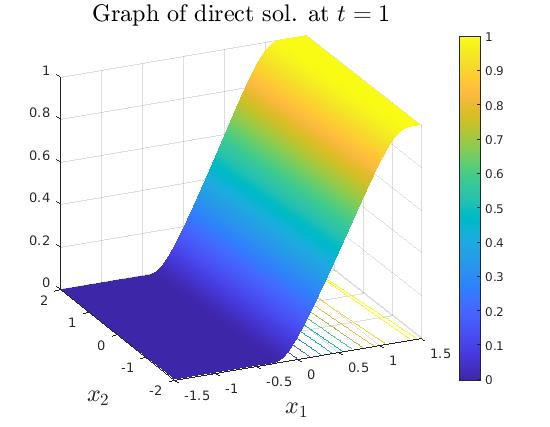}
\end{center}
\caption{{\bf Experiment 2.} 
(Top): Solution ${\bf u}$ with the method of Subsection \ref{ss:1xN,rw},
(Left) at $t=1/4$, (Right) at $t=1$. 
(Bottom): Solution with PINN-method,
(Left) at $t=1/4$, (Right) at $t=1$.}
\label{fig2a}
\end{figure}
We also report in Fig. \ref{fig2b} (Left), the loss functions (constructed using the $H^1$-norm on $\Lambda_0=(-2,2)\times (0,1)$ with self-similar neural networks) as well as the loss function with the direct PINN-method, to illustrate the convergence of the optimization algorithm.
\begin{figure}[hbt!]
\begin{center}
  \includegraphics[height=6cm,keepaspectratio]{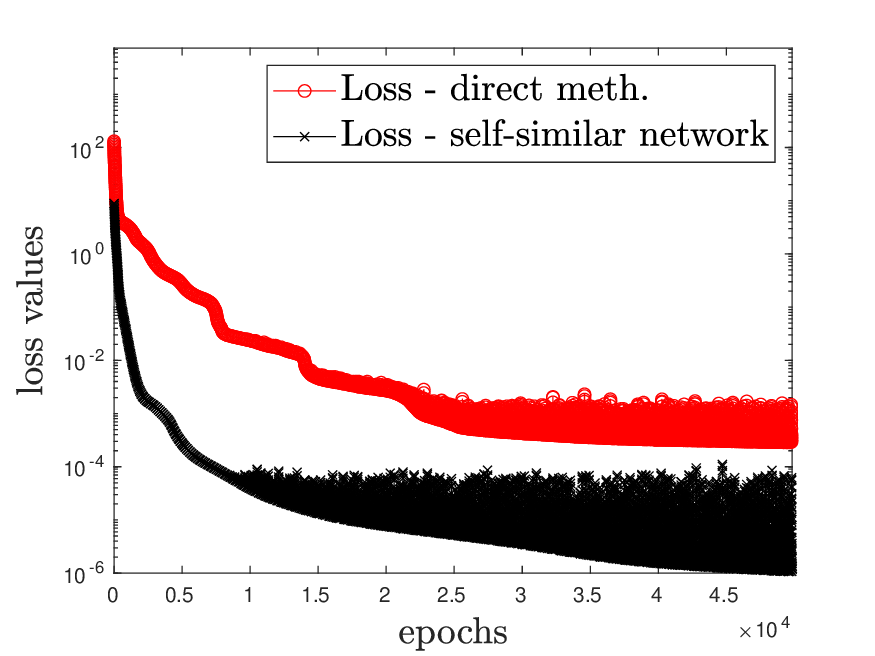}
   \includegraphics[height=6cm,keepaspectratio]{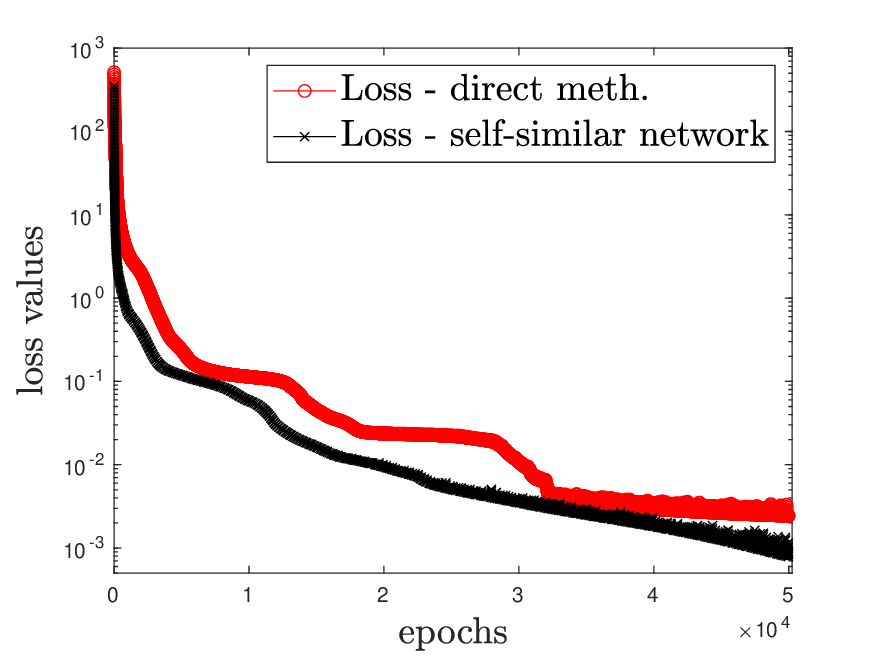} 
\end{center}
\caption{{\bf Experiment 2.} (Left) Loss functions with direct PINN-method and with self-similar networks. {\bf Experiment 2 bis.} (Right) Loss functions with direct PINN-method and with self-similar networks.}
\label{fig2b}
\end{figure}

\medskip

\noindent{\bf Experiment 2 bis. (1-D rarefaction wave)} 
In this experiment we consider the same data as in Experiment 2, excepted that we change the flux function $F=(F_1,F_2)$, by $F_1(u)=u^4/4$, $F_2(u)=u^2/8$. We report in Fig. \ref{fig2c} the solution $u_{\textrm{rw}}$ at $t=1/4$ (Top-Left) and at $t=T=1$ (Top-Right). Similarly, we report the corresponding  solutions obtained with a direct PINN method and at $t=1/4$ (Bottom-Left) and $t=1$ (Bottom-Right). 
We use the same number of layers and neurons in both methods.

\begin{figure}[hbt!]
\begin{center}
  \includegraphics[height=5.5cm,keepaspectratio]{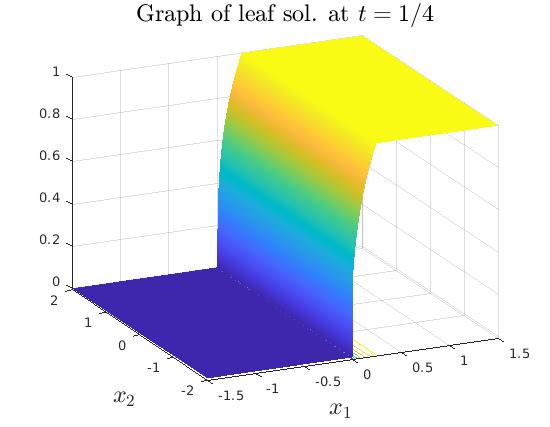}
  \includegraphics[height=5.5cm,keepaspectratio]{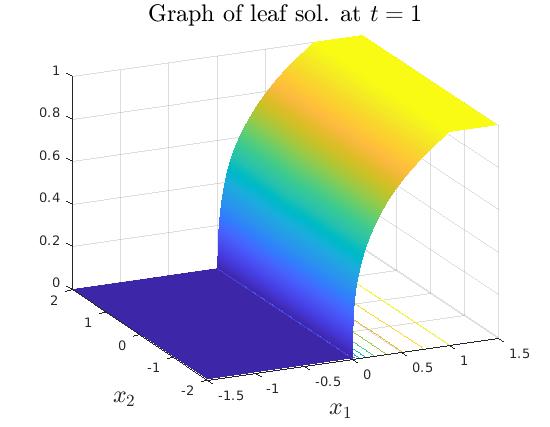}\\
    \includegraphics[height=5.5cm,keepaspectratio]{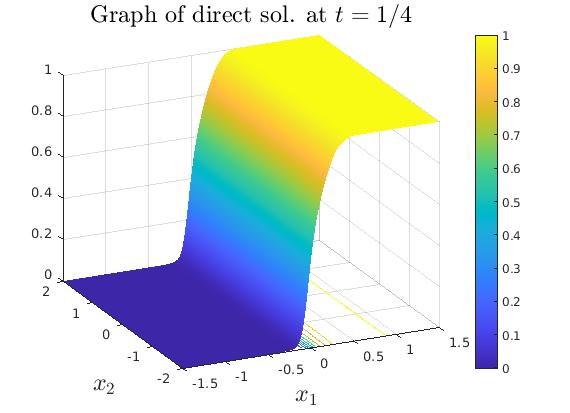}
  \includegraphics[height=5.5cm,keepaspectratio]{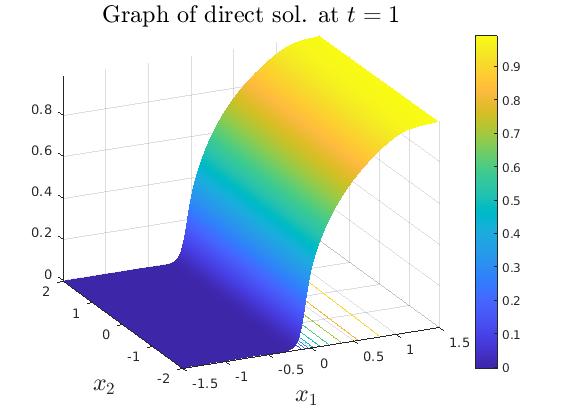}
\end{center}
\caption{{\bf Experiment 2 bis.} (Top-Left) Solution with self-similar networks at $t=1/4$. (Top-Right) Solution with self-similar networks at $t=1$. (Bottom-Left) Solution with PINN-method at $t=1/4$. (Bottom-Right) Solution with PINN-method at $t=1$.}
\label{fig2c}
\end{figure}
For completeness, we also report in Fig. \ref{fig2c} (Right), the loss functions (constructed using the $H^1$-norm on $\Lambda_0=(-2,2)\times (0,1)$ with self-similar neural networks) as well as the loss function with the direct PINN-method, to illustrate the convergence of the optimization algorithm. 

These numerical experiments illustrate the accuracy of the proposed strategy to approximate rarefaction and shock waves.

\medskip

\noindent{\bf Experiment 3 (2-D rarefaction wave)} 
This experiment is dedicated to the approximation of 2-D rarefaction wave using the strategy described in Section \ref{ss:1xN,rw}, with a neural network having $3$ hidden layers and $15$ neurons each. The networks are trained with $300$ randomly chosen collocation points in $x_1$, $x_2$ and $t$.  We consider the following Riemann problem on $Q=\Omega\times [0,T]$, where $\Omega=(-2,2)\times(-2,2)$ and $T=1.1$. The initial data is reported in Fig. \ref{fig3new} (Left)
\begin{eqnarray*}
  u_0(x_1,x_2)=\left\{
  \begin{array}{ll}
    u_0^l, & x_1< g_0(x_2),\\
    u_0^r, & x_1> g_0(x_2) \, , 
  \end{array}
  \right.
\end{eqnarray*}
with $u_0^l=0$ and $u_0^r=1$, $g_0(x_1,x_2)=(x_2+2)^2/10 -1$, and the flux function $F=(F_1,F_2)$ is defined by $F_1(u)=u^2/2$, $F_2(u)=u^2/8$.  We report in Fig. \ref{fig3new} the solution $u_{\textrm{rw}}$ at $t=T=1.1$ (Right). This example illustrates that the use of self-similar neural networks allows for an accurate approximation of rarefaction waves with no artificial regularization of non-smooth regions.
\begin{figure}[hbt!]
  \begin{center}
    \includegraphics[height=5.5cm,keepaspectratio]{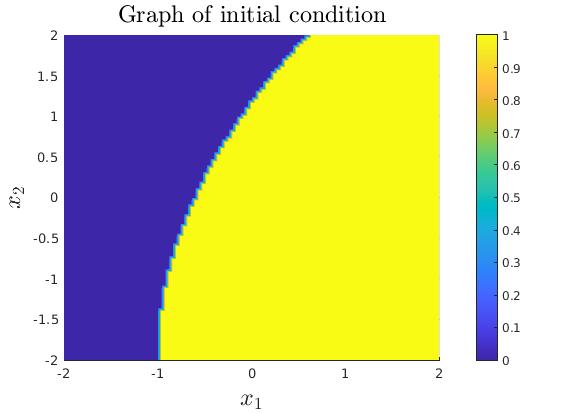} 
  \includegraphics[height=5.5cm,keepaspectratio]{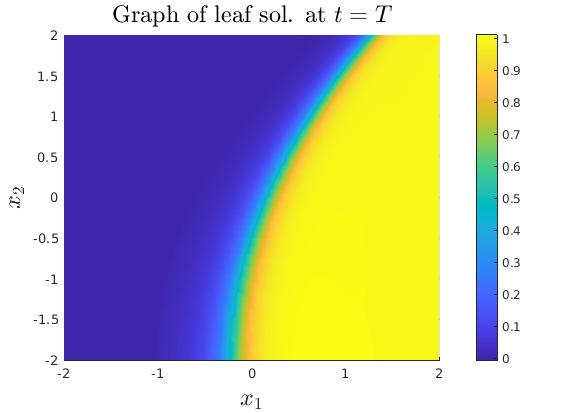}
\end{center}
\caption{{\bf Experiment 3.} (Left) Initial data. (Right) Solution with self-similar networks at $t=T$.}
\label{fig3new}
\end{figure}

\medskip

\noindent{\bf Experiment 4 (2-D shock formation)} 
This experiment is an illustration of a shock generation from a smooth initial condition. 
We compare the solution obtained with a direct PINN-method with the algorithm proposed in Subsection \ref{ss:1xN,shock,form}. We consider $Q=\Omega\times [0,T]$, with $\Omega=(-1,1)\times(-2,2)$ and $T=0.8$. The flux function $F=(F_1,F_2)$ with $F_1(u)=u^2$, $F_2(u)=u^2/4$ and the initial condition is taken as
\begin{eqnarray*}
u_0(x_1,x_2) = \cfrac{1}{2}(1-\tanh(5(x_1 + (x_2+2)^2/20))) \, .
  \end{eqnarray*}
In this example, the solution is smooth up to time $t=t^*$ corresponding to the generation of a shock wave:
\begin{eqnarray}\label{Tstar2}
  \left.
  \begin{array}{lcl}
    t^{*}  & = & -\big(\textrm{min}_{(x_1^0,x_2^0) \in \Omega} \sum_{i=1}^2 \partial_{x^0_i}(f_i(u_0(x_1,x_2)))\big)^{-1} \, ,
\end{array}
\right.
\end{eqnarray}
which is equal to $0.2$ in the considered example. The algorithm from Subsection \ref{ss:1xN,shock,form} is applied with $2$ hidden layers with $15$ neurons each.  In Fig. \ref{figRec1} (Left) we plot the initial data $u_0$ and the solution at $t=t^*$ (Right). In Fig. \ref{figRec1b}, we report the PINN solution and the LeafNet solution (Right) at $t=T$.
\begin{figure}[hbt!]
  \begin{center}  \includegraphics[height=6cm,keepaspectratio]{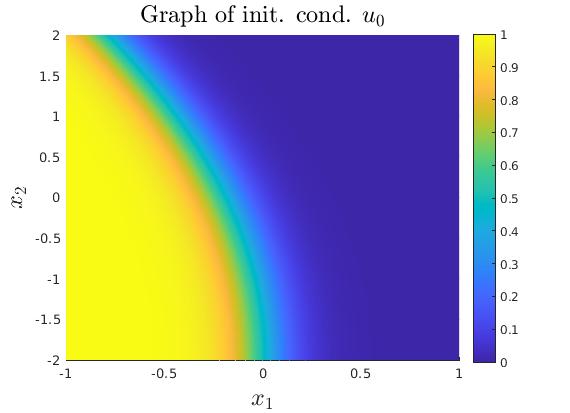}
           \includegraphics[height=6cm,keepaspectratio]{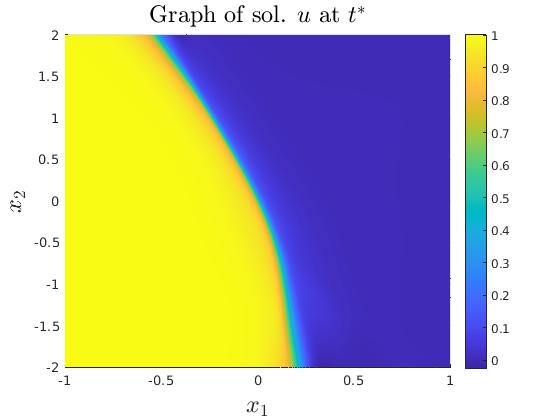} 
\end{center}
\caption{{\bf Experiment 4.} (Left) Graph of  $u_0$. (Right) Graph of the solution at $t=t^*$.}
\label{figRec1}
\end{figure}
\begin{figure}[hbt!]
  \begin{center}
        \includegraphics[height=6cm,keepaspectratio]{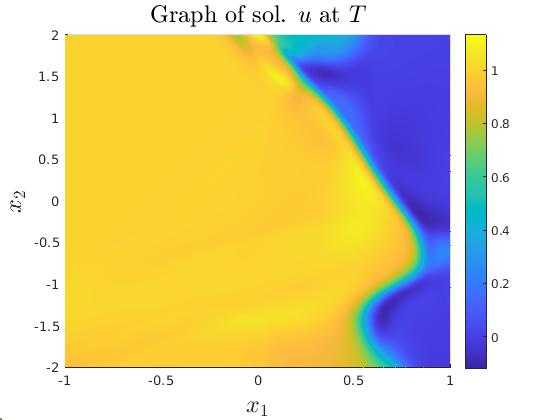}
  \includegraphics[height=6cm,keepaspectratio]{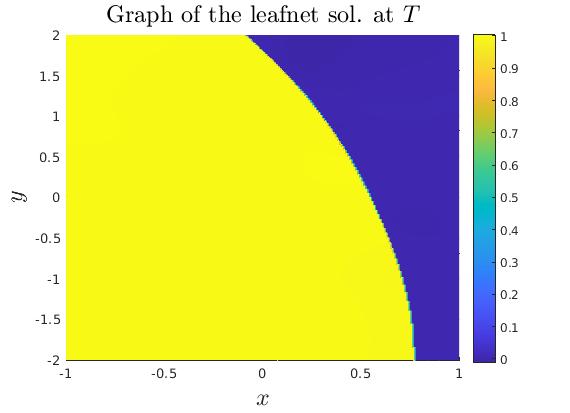}
\end{center}
\caption{{\bf Experiment 4.} (Left) Graph of the PINN solution at $t=T$. (Right) Graph of the LeafNet solution at $t=T$.}
\label{figRec1b}
\end{figure}

 We report in Fig. \ref{figRec2} (Left) the surface of discontinuity, and the loss function for the LeafNet algorithm  in Fig. \ref{figRec2} (Right). 
\begin{figure}[hbt!]
\begin{center}
    \includegraphics[height=6cm,keepaspectratio]{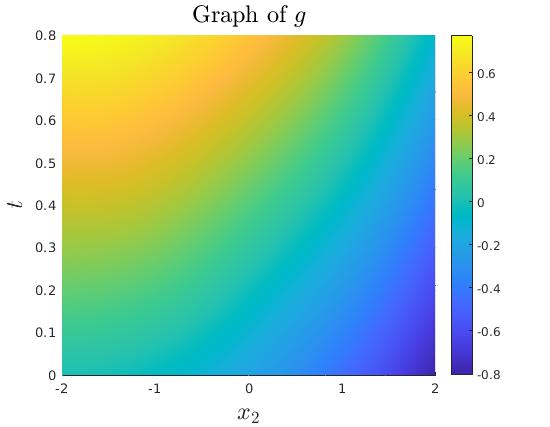}
      \includegraphics[height=6cm,keepaspectratio]{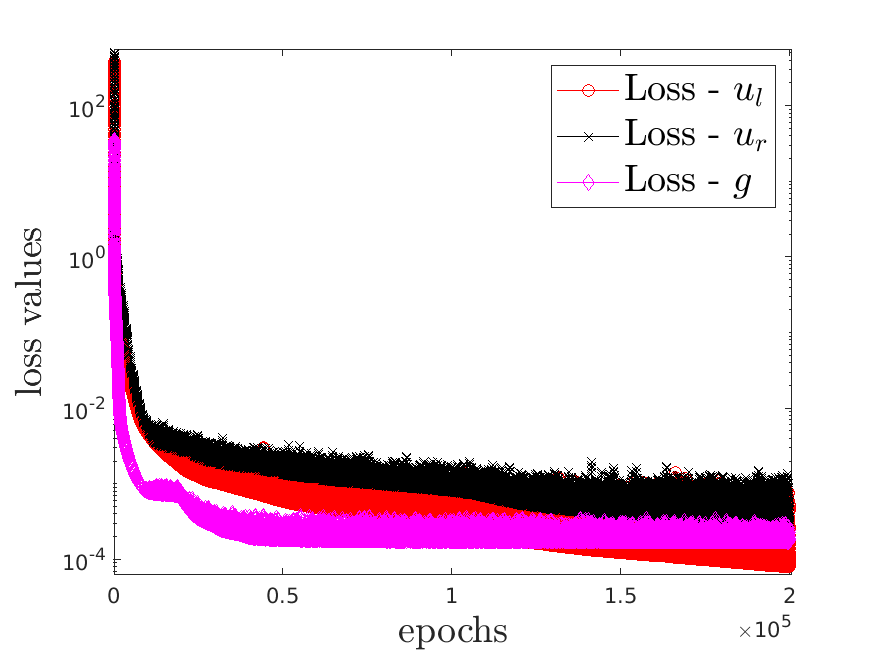}
\end{center}
\caption{{\bf Experiment 4.} (Left)  Graph of the surface of discontinuity. (Right) Loss functions.}
\label{figRec2}
\end{figure}
This test illustrates the ability of the algorithm to accurately generate and propagate shock waves.

\subsection{Systems of HCL}
In this subsection, we propose a preliminary illustration of the above strategy applied to a single shock-wave solution to a two-dimensional system.  We consider here the two-dimensional shallow-water equations with flat bottom modeling compressible fluid flows, and corresponding to $m=3$ in \eqref{e:CL(N,m)}. 
The latter models the evolution of a fluid velocity field $(\mathcal{u},\mathcal{v})$ and height $\mathcal{h}$ in shallow water.  At this stage, the proposed method does not consider initial wave decompositions for instance from an arbitrary discontinuous initial data (or more complex initial data), and is mainly designed for simple waves (rarefaction and$/$ shock waves. For this reasons we here focus on a relatively simple case (one shock wave). The conservative variable reads 
$u=(u_1,u_2,u_3)=(\mathcal{h},\mathcal{h}\mathcal{u},\mathcal{h}\mathcal{v})$, and the flux functions are defined as $F_k=(F_k^i)$, $k=1,2$, $i=1,2,3$, where
\[
\begin{array}{rllllllll}
F_k^*(u)
&= 
(\mc{hu}, &\mc{hu^2 + gh^2/2}, &\mc{huv}) 
&=(u_2,&u_2^2/u_1 + gu_1^2/2,& u_2u_3/u_1), \\
F_2^*(u)
&=(\mc{hv},&\mc{huv},&\mc{hv^2+gh^2/2})
&=(u_3,&u_2u_3/u_1,          &u_3^2/u_1 + gu_1^2/2).
\end{array}
\]
The system is strictly hyperbolic for $\mathcal{h}>0$, and the eigenvalues of 
$\nabla F_1^*$ (written in non conservative variables) are 
\[
\lambda^1_1 = \mathcal{u}-\sqrt{\mathcal{gh}},\;\;
\lambda^2_1 = \mathcal{u}+\sqrt{\mathcal{gh}},\;\;
\lambda^3_1 = \mathcal{u},
\]
with corresponding eigenvectors 
\[
v^1_1=(1,\mathcal{u}-\sqrt{\mathcal{gh}},\mathcal{v}),\;\;
v^2_1=(0,0,1),\;\;
v^3_1=(1,\mathcal{u}+\sqrt{\mathcal{gh}},\mathcal{v}).
\] 
The eigenvalues of $\nabla F^*_2$ read  (written in non conservative variables) 
\[
\lambda^1_2 = \mathcal{v}-\sqrt{\mathcal{gh}},\;\;
\lambda^2_2 = \mathcal{v},\;\;
\lambda^3_2 = \mathcal{v}+\sqrt{\mathcal{gh}},
\]
and corresponding eigenvectors are given by 
\[
v^1_2=(1,\mathcal{u},\mathcal{v}-\sqrt{\mathcal{gh}}),\;\;
v^2_2=(0,1,0),\;\;
v^3_2=(1,\mathcal{u},\mathcal{v}+\sqrt{\mathcal{gh}}).
\]
Setting $\omega=(\omega_1,\omega_2) \in \R^2$ such that $|\omega|=1$,  the real eigenvalues of $\omega_1\nabla F^*_1 +\omega_2\nabla F^*_2$ are given by 
$\ell_1(\omega)=\omega_1 \mathcal{u} + \omega_2 \mathcal{v} -\sqrt{g\mathcal{h}}$, $\ell_2(u,\omega)=\omega_1 \mathcal{u} + \omega_2 \mathcal{v}$, $\ell_3(\omega)=\omega_1 \mathcal{u} + \omega_2 \mathcal{v} + \sqrt{g\mathcal{h}}$. 
It will be convenient to solve the IBVP associated to $(u_1,u_2,u_3)$ in non conservative form, which in terms of $\mc{h, u, v}$ reads as
\begin{eqnarray}\label{sw}
        \left.
      \begin{array}{rcl}
        \partial_t \mathcal{h} + \mathcal{h}(\partial_{x_1}\mathcal{u} + \partial_{x_2}\mathcal{v}) + \mathcal{u}\partial_{x_1} \mathcal{h} + \mathcal{v}\partial_{x_2}\mathcal{h} & = & 0 \, ,\\
        \partial_t \mathcal{u}  + \mathcal{u}\partial_{x_1}\mathcal{u} + \mathcal{v}\partial_{x_2}\mathcal{u}+ g\partial_{x_1} \mathcal{h}  & = & 0 \, ,    \\
         \partial_t \mathcal{v}  + \mathcal{u}\partial_{x_1}\mathcal{v} + \mathcal{v}\partial_{x_2}\mathcal{v}+ g\partial_{x_2} \mathcal{h}  & = & 0 \, .              
         \end{array}
      \right.       
\end{eqnarray}
We assume here that the initial discontinuity is a straight line $\Gamma_0$ with equation
${g}_0(x_2)=\beta x_2$, where $\beta=1/2$. 
We design a relevant non-trivial initial condition, we ensure that the left and right states at the discontinuity satisfy the Rankine-Hugoniot jump condition. In this goal, we first consider two {\it constant} conservative states 
$\overline{s}^l_0=(\overline{\mathcal{h}}_0^l,\overline{\mathcal{h}}_0^l\overline{\mathcal{u}}^l_0,\overline{\mathcal{h}}_0^l\overline{\mathcal{v}}^l_0)$ and 
$\overline{s}_0^r=(\overline{\mathcal{h}}_0^r,\overline{\mathcal{h}}_0^r\overline{\mathcal{u}}^l_0,\overline{\mathcal{h}}_0^r\overline{\mathcal{v}}^r_0)$,
separated by the oblique shock $\Gamma_0$ with unit normal vector $\omega=(\omega_1,\omega_2)$,
such that for some $\sigma^0$, we have
\begin{eqnarray}\label{tmpRH}
  \sigma^0\jump{\overline{s}^0}
  =
  \omega_1 \jump{F_1(\overline{s}^0)} + 
  \omega_2 \jump{F_2(\overline{s}^0)}.
\end{eqnarray}
We assume here that the initial discontinuity is a straight line with ${g}_0(x_2)=\beta x_2$, where $\beta=-\omega_1/\omega_2$, so that from \eqref{tmpRH} we get
\begin{eqnarray*}
  \cfrac{\sigma^0}{\omega_1}\jump{\overline{s}^0}
  =  
  \jump{F_1(\overline{s}^0)}
  -\beta \jump{F_2(\overline{s}^0)}.
\end{eqnarray*}
Explicitly,  we take $\omega_1=1/\sqrt{5}$, $\omega_2=-2/\sqrt{5}$ corresponding to $\beta=1/2$. 
After simple preliminary calculations, we select: 
$\overline{\mathcal{h}}_0^l=2$, $\overline{\mathcal{h}}_0^r=1$, $\overline{\mathcal{u}}_0^l=\sqrt{2g/5}$, $\overline{\mathcal{u}}_0^r=0$, $\overline{\mathcal{v}}_0^l=-\sqrt{g/10}$, $\overline{\mathcal{v}}_0^r=0$ and $\sigma^0=\sqrt{5}\overline{\mathcal{u}}_0^l/2=\sqrt{g/2}$. The corresponding eigenvalues of 
$\omega_1 \nabla F_1(\overline{s}_{\nu}^0) + 
\omega_2 \nabla F_2(\overline{s}_{\nu}^0)$ are given by
\begin{eqnarray*}
  \left.
  \begin{array}{ll}
    \ell_{1}(\overline{s}^{l}_0)=\omega_1 \big(\sqrt{2g/5}-1\big) - \omega_2 \big(\sqrt{g/10}+1\big), & \ell_{1}(\overline{s}^{r}_0)= -\omega_1 - \omega_2,\\
     \ell_{2}(\overline{s}^{l}_0)=\omega_1\sqrt{2g/5} -\omega_2\sqrt{g/10}, &\ell_{2}(\overline{s}^{r}_0)= 0,\\   
     \ell_{3}(\overline{s}^{l}_0)=\omega_1 \big(\sqrt{2g/5}+1\big) + \omega_2 \big(1-\sqrt{g/10}\big), & \ell_{3}(\overline{s}^{r}_0)= \omega_1 + \omega_2 \, .  
  \end{array}
  \right.
  \end{eqnarray*}
\noindent{\it Characterization of the shock.} Selecting $g=1$, and $\omega=(2,-1)/\sqrt{5}$, we then have $\sigma^0=1/\sqrt{2} \approx 0.7071$ and $\ell_1(\overline{s}^{l}_0)\approx 1.013$, $\ell_1(\overline{s}^{r}_0) \approx 0.447$, that is  $\ell_1(\overline{s}^{r}_0) <\sigma^0 < \ell_1(\overline{s}^{l}_0)$ corresponding to a 1-shock.
\\

Based on those preliminary computations we construct the initial data in non conservative form 
$(\mathcal{h}^r_0,\mathcal{u}_0^r,\mathcal{v}^r_0)=(1,0,0)$ and  
$(\mathcal{h}^l_0,\mathcal{u}_0^l,\mathcal{v}^l_0)$ in $\Omega$ defined by
\begin{eqnarray*}
\mathcal{h}_0^l(x_1,x_2) 
&=&
(\overline{\mathcal{h}}^l_0 -\zeta k_1(\min\{x_1,{g}_0(x_2)\},x_2))
k_2(\min\{x_1,{g}_0(x_2)\},x_2),
\\
\mathcal{u}_0^l(x_1,x_2) 
&=&
(\overline{\mathcal{u}}^l_0 -\zeta k_1(\min\{x_1,{g}_0(x_2)\},x_2))k_2
(\min\{x_1,{g}_0(x_2)\},x_2),
\\
\mathcal{v}_0^l(x_1,x_2) 
&=&
(\overline{\mathcal{v}}^l_0 k_1(\min\{x_1,{g}_0(x_2)\},x_2))k_2(\min\{x_1,{g}_0(x_2)\},x_2),
\end{eqnarray*}
and where
\begin{eqnarray*}
 k_1(x_1,x_2)={g}_0(x_2)-(x_1-(a_1+b_1)/2), \, \, \,    k_2(x_1,x_2)= \gamma \tanh(\alpha k_1(x_1,x_2)) \, .
\end{eqnarray*}
\vspace*{1mm}

\medskip

\noindent{\bf Experiment 5 (2-D system, shock)} 
Here we take $a_1=-1$, $b_1=2$, $a_2=-b_2=-1$, $\alpha=20$, $\zeta=1/4$, $\gamma=1/2$. The final $T$ is fixed to $0.5$. In the numerical experiment, the networks for approximating the IBVP on  $(u_1,u_2,u_3)$ and ${g}$ have 2 hidden layers with $10$ neurons. The number of randomly chosen collocation points is $150$ in $x_1$, $x_2$ and $t$.
We report  in Fig. \ref{fig3a} the graphs of $\mathcal{h}^0$, $\mathcal{u}^0$ and $\mathcal{v}^0$. We then report in Fig. \ref{fig3b}, the graphs of the reconstructed solution at final time $T$ (non conservative variables) $\mathcal{h}(\cdot,T)$, $\mathcal{u}(\cdot,T)$ and  $\mathcal{v}(\cdot,T)$. We report in Fig. \ref{fig3c} (Left) the graph of ${g}$ in $Q_2$ defining the time-space surface of discontinuity. For completeness, we also report the graph of convergence of the minimization algorithms in Fig. \ref{fig3c} (Right).
\begin{figure}[hbt!]
\begin{center}
  \includegraphics[height=3.95cm,keepaspectratio]{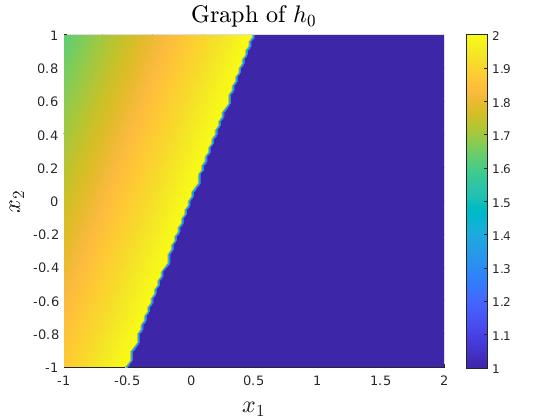}
   \includegraphics[height=3.95cm,keepaspectratio]{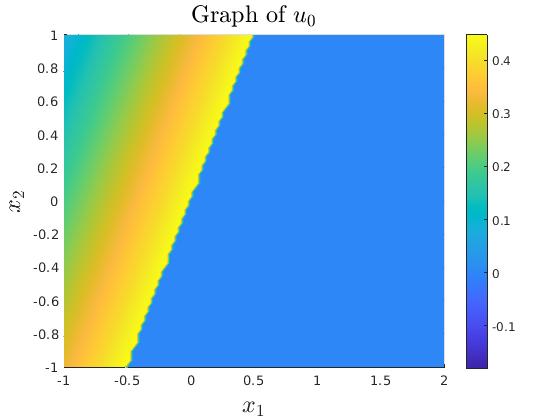}
  \includegraphics[height=3.95cm,keepaspectratio]{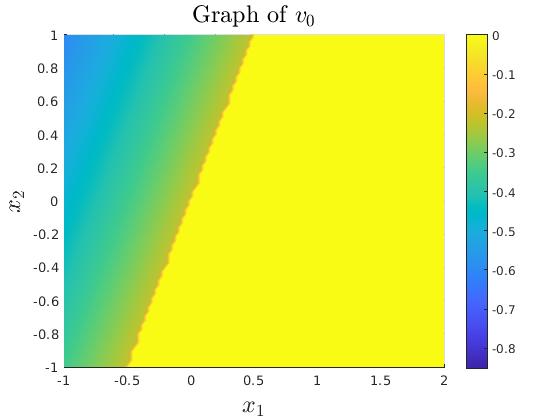}
\end{center}
\caption{{\bf Experiment 5.}  (Left) Graph of  $\mathcal{h}_l$. (Middle) Graph of $\mathcal{u}^0$. (Right) Graph of $\mathcal{v}^0$.}
\label{fig3a}
\end{figure}

\begin{figure}[hbt!]
\begin{center}
  \includegraphics[height=3.95cm,keepaspectratio]{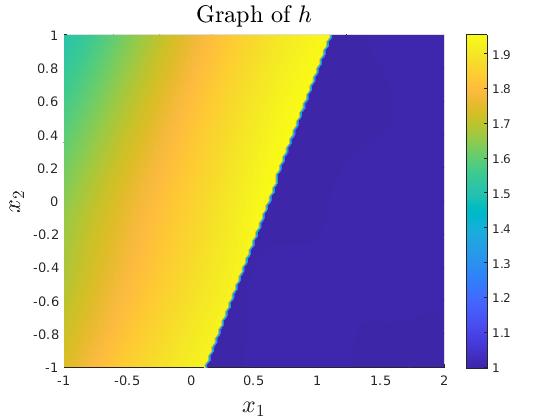}
  \includegraphics[height=3.95cm,keepaspectratio]{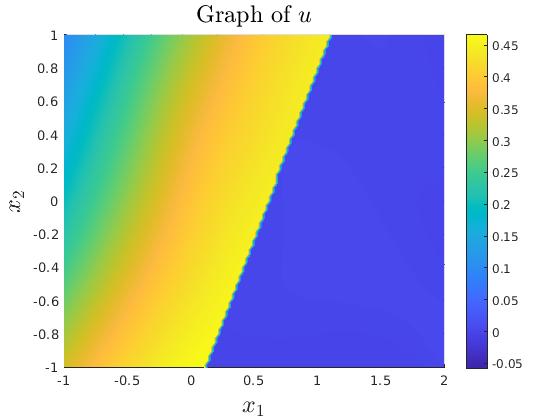}
  \includegraphics[height=3.95cm,keepaspectratio]{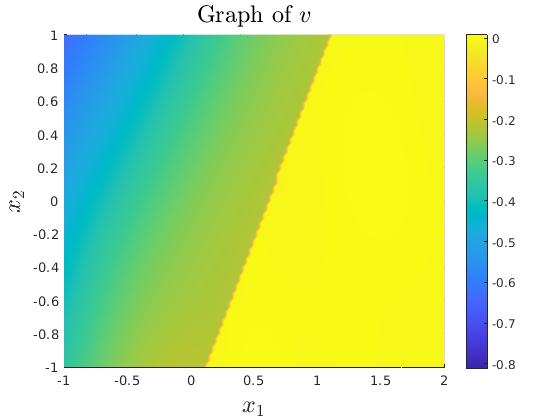}
\end{center}
\caption{{\bf Experiment 5.} At final time $T$ (Left) Graph of  
${\boldsymbol{\mathcal h}}(\cdot,\cdot,T)$. (Middle) Graph of $\boldsymbol{\mathcal{u}}(\cdot,\cdot,T)$.  (Right) Graph of $\boldsymbol{\mathcal{v}}(\cdot,\cdot,T)$.}  
\label{fig3b}
\end{figure}

\begin{figure}[hbt!]
\begin{center}
  \includegraphics[height=6cm,keepaspectratio]{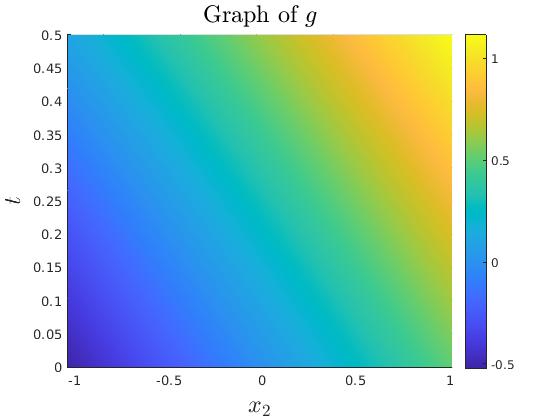}
   \includegraphics[height=6cm,keepaspectratio]{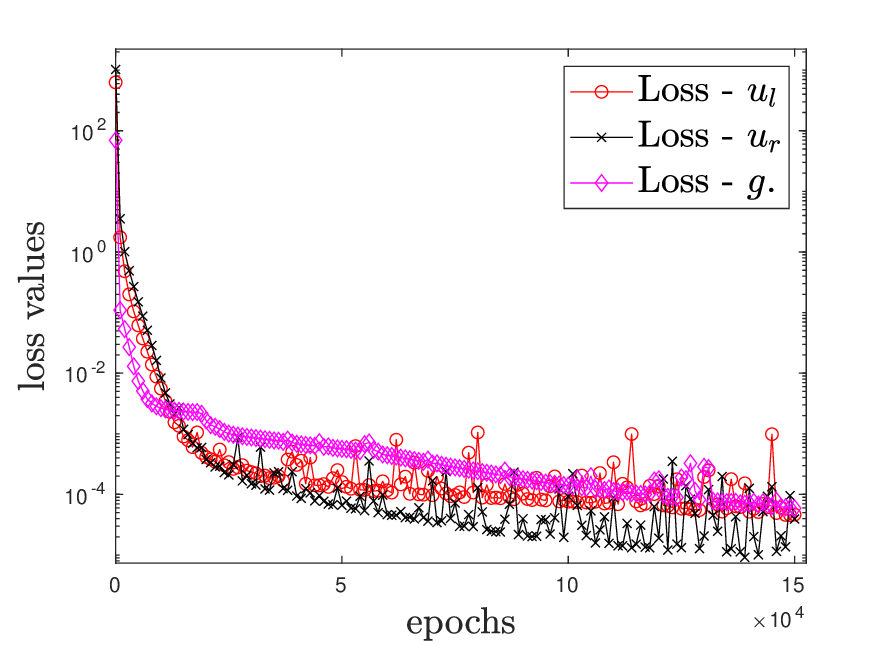} 
\end{center}
\caption{{\bf Experiment 5.} (Left) Graph of ${\bf g}$. (Right) Loss functions for training 
${\bf u}^l$, ${\bf u}^r$ and ${\bf g}$.}
\label{fig3c}
\end{figure}
We yet observe that the propagation of the shock wave is performed without any artificial diffusion,
which is expected from the design of the algorithm.

\medskip

\noindent{\bf Experiment 5 bis (2-D system, shock)} 
In this following experiment, we consider the exact same setting as above except that we now consider a non-straight initial discontinuity:  ${g}_0(x_2)=\beta(x_2-a_2)$, where $\beta=10^{-1}$. Here, we report the initial and final LeafNet solution in Fig. \ref{fig4b}.
\begin{figure}[hbt!]
  \begin{center}
  \includegraphics[height=3.95cm,keepaspectratio]{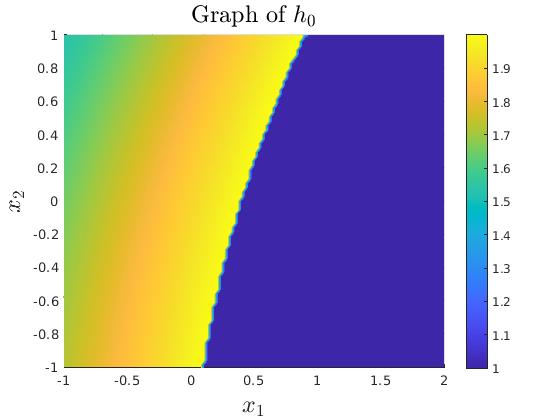}
   \includegraphics[height=3.95cm,keepaspectratio]{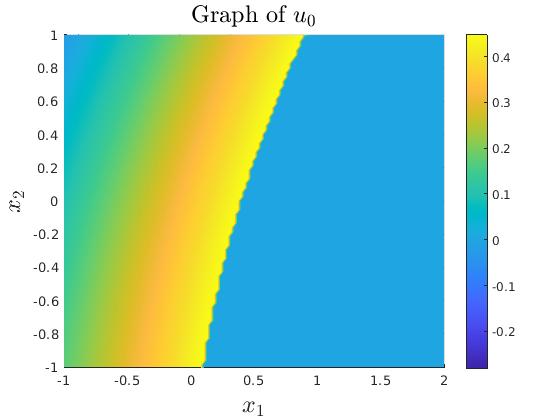}
  \includegraphics[height=3.95cm,keepaspectratio]{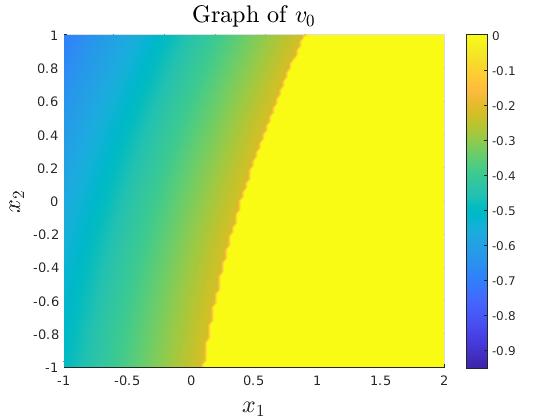}\\    
  \includegraphics[height=3.95cm,keepaspectratio]{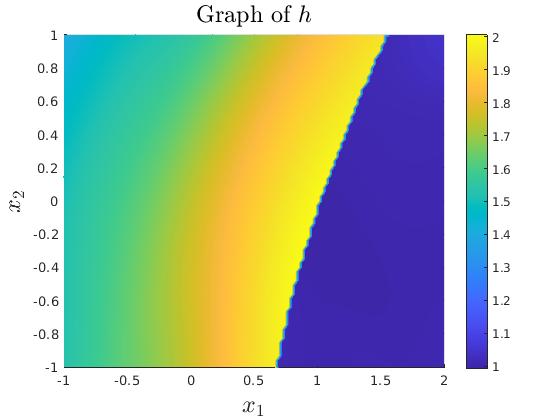}
  \includegraphics[height=3.95cm,keepaspectratio]{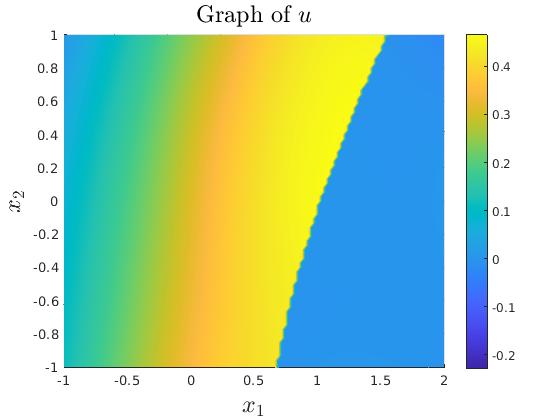}
  \includegraphics[height=3.95cm,keepaspectratio]{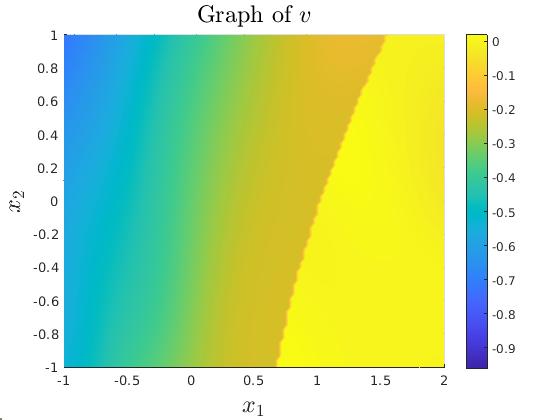}
\end{center}
\caption{{\bf Experiment 5 bis.} Initial data (Top) LeafNet solution at final time $T$ (Bottom)At final time $T$.}
\label{fig4b}
\end{figure}
We also report the corresponding loss function and graph of ${\bf g}$ in Fig. \ref{fig4c}
\begin{figure}[hbt!]
\begin{center}
  \includegraphics[height=6cm,keepaspectratio]{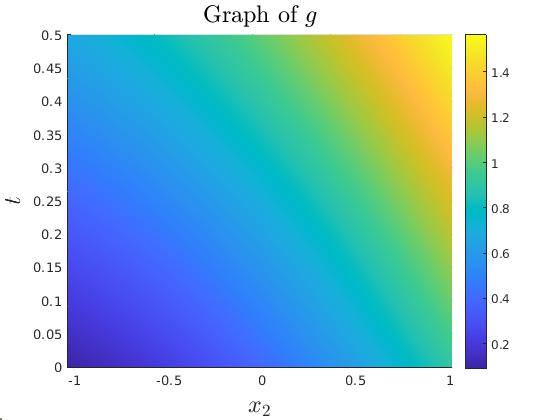}
   \includegraphics[height=6cm,keepaspectratio]{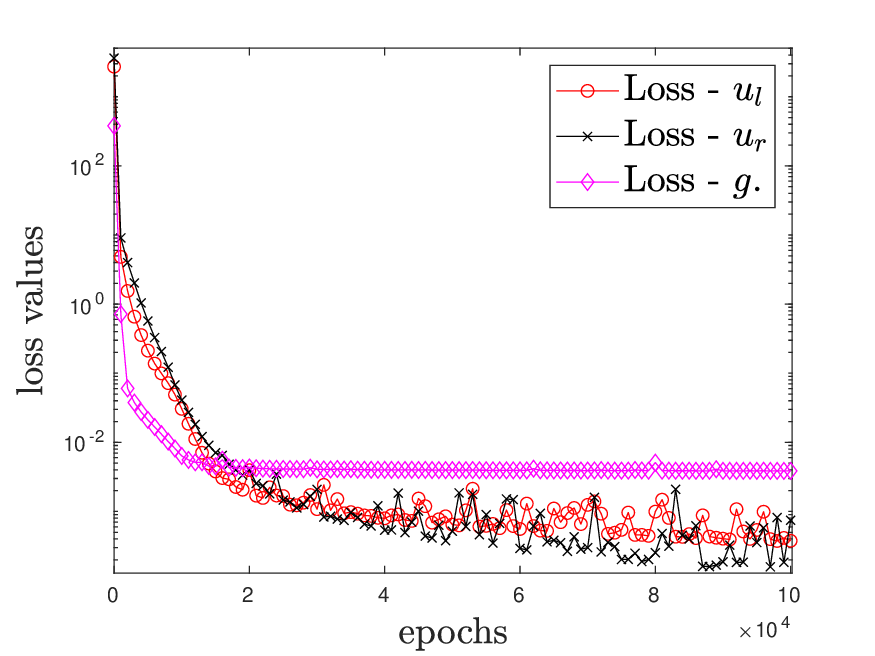} 
\end{center}
\caption{{\bf Experiment 5 bis.} (Left) Graph of ${\bf g}$. (Right) Loss functions for training 
${\bf u}^l$, ${\bf u}^r$ and ${\bf g}$.}
\label{fig4c}
\end{figure}

This test illustrates that the method can as well as non-oblique shock.

\color{black}
\subsection{About the entropy of the LeafNet method}
The LeafNet method does not explicitly impose any entropy constraint. In the case of a convex flux, the physically relevant entropy solution can be identified relatively easily, since admissible shocks satisfy the classical Lax entropy condition. For non-convex fluxes, however, the structure of the entropy solution is generally more intricate, as Riemann solutions may involve several intermediate states and composite wave patterns consisting of shocks and rarefaction waves.
Hence, a discontinuous wave between two states does not necessarily generate a single Lax shock or a single rarefaction wave. The corresponding entropy solution may instead contain compound waves, such as shock-rarefaction or rarefaction-shock configurations, with intermediate states determined by tangent or convex/concave envelope constructions. 
The LeafNet method is able to solve these shock-rarefaction or rarefaction-shock configurations with intermediate states, see Section \ref{ss:1xN,rw}, 
without any distinction between entropic or non-entropic solution because the present LeafNet formulation does not explicitly enforce any entropy selection mechanism. However, once the appropriate entropy conditions have been specified for a given non-convex flux---a question that is independent of the numerical solver---they can be readily incorporated into the LeafNet framework; see Section \ref{ss:1xN,rw}.
\color{black}

\medskip

\noindent{\bf Experiment 6.} To illustrate the inability of LeafNet to directly capture entropy solutions with non-convex cases, we propose a simple one-dimensional equation
\begin{eqnarray*}
u_t + f(u)_x = 0\, ,
\end{eqnarray*}
where the flux function is given by $f(u)=u^3-u$ and initial data:
\begin{eqnarray*}
  u_0(x) =
  \left\{
  \begin{array}{cc}
    1, & x<0, \\
    -1, & x>0 \, .
    \end{array}
    \right.
\end{eqnarray*}
The entropy solution is a compound wave consisting of a shock from $u_L=1$ to an intermediate state $u_*=-1/2$, followed by a rarefaction wave from $u_*=-1/2$ to $u_R=-1$. The intermediate state is determined by the tangency condition
\begin{eqnarray*}
f'(u_*)=\cfrac{f(u_*)-f(1)}{u_*-1} \, .
\end{eqnarray*}
The exact solution is given by 
\begin{eqnarray*}
u(x,t) =
\begin{cases}
1, & \cfrac{x}{t}<-\cfrac{1}{4}, \\
     -\sqrt{\cfrac{1+x/t}{3}}, & -\cfrac{1}{4}<\cfrac{x}{t}<2, \\
-1, & \cfrac{x}{t}>2 \, .
\end{cases}
\end{eqnarray*}
In Fig. \ref{figNC}, we report in space-time, the exact solution (Left) and the numerical solution (Right)  obtained by the direct application of the leafneat method considering only one shock wave. As expected the latter is not able to directly capture the exact solution which was designed specifically for convex fluxes.
\begin{figure}[hbt!]
\begin{center}
  \includegraphics[height=5cm,keepaspectratio]{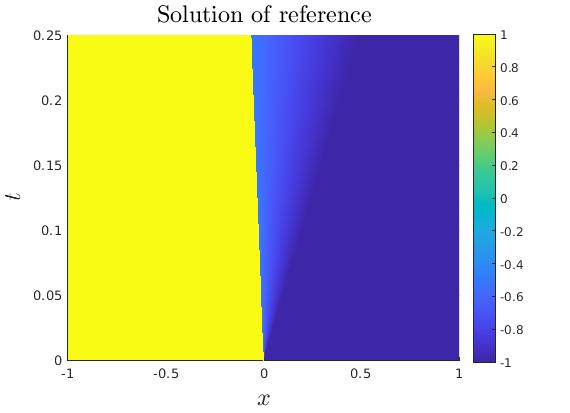}
  \includegraphics[height=5cm,keepaspectratio]{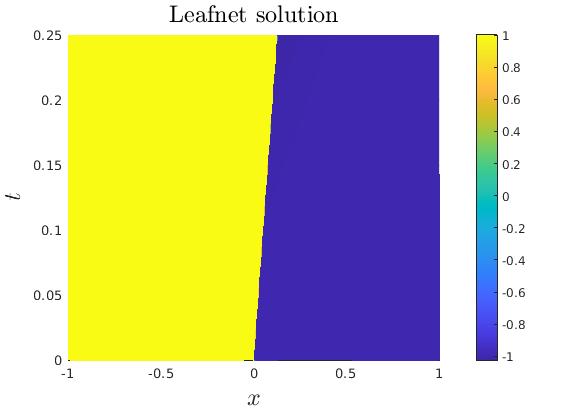}
\end{center}
\caption{{\bf Experiment 6.} Graph of the solution in space-time. (Left) Reference solution. (Right) LeafNet solution.}
\label{figNC}
\end{figure}

\color{black}

\section{Conclusion}\label{s:conclusion}
We have presented a neural network based computational method for accurately solving piece-wise smooth solutions to $N$-dimensional hyperbolic conservation laws, and more particularly for tracking shock waves, formation of shock waves and rarefaction waves. In the case of shock waves, we have shown rigorously that the associated method is well-posed and we have given an error estimated of the shock wave approximation. We have also shown that the proposed method can be extended to the formation and evolution of shock waves solution to $N$-dimensional systems of conservation laws.

The driving idea of the method is the reformulation of the HCL as a coupled system involving only smooth local solutions, hence providing a mathematical setting compatible with the use of neural network algorithms approximating PDE in their strong form. 
We then have developed a neural network-based algorithm for approximating the new equivalent coupled system. Several numerical experiments were also proposed to illustrate the proposed strategy and its accuracy to approximate shock and rarefaction waves.

\bibliographystyle{unsrt}
\bibliography{refs_leaf}

\end{document}